\documentclass[a4paper, reqno, 11pt]{amsart}
    \usepackage[utf8]{inputenc} 
    \usepackage{fullpage}
    \usepackage[american]{babel}
    \usepackage{amssymb,graphicx,color,eucal,stmaryrd}
    \usepackage[shortlabels]{enumitem} 
    \usepackage[all,cmtip]{xy}
    
    \usepackage{mathtools}
    
    \usepackage{hyperref} 
    \definecolor{darkblue}{rgb}{0,0,.85} 
    \definecolor{darkred}{rgb}{0.84,0,0}
    \hypersetup{colorlinks = true,
            linkcolor = darkblue,
            urlcolor  = darkblue,
            citecolor = darkred,
            anchorcolor = darkblue}
            
    \SetLabelAlign{LeftAlignWithIndent}{\hspace*{2.0ex}\makebox[1.5em][l]{#1}}
        	
    \numberwithin{equation}{subsection}
    
    \makeatletter
    \def\paragraph{\@startsection{paragraph}{4}%
    \z@\z@{-\fontdimen2\font}%
    {\normalfont\bfseries}}
    \makeatother
    
    \newdir{ >}{{}*!/-5pt/@{>}}
    \SelectTips{cm}{10}

    \newtheorem{lem}{Lemma}[section]
    \newtheorem{cor}[lem]{Corollary}
    \newtheorem{thm}[lem]{Theorem}
    \newtheorem{prop}[lem]{Proposition}
    
    \theoremstyle{definition}
    \newtheorem{definition}[lem]{Definition}
    \newtheorem{construction}[lem]{Construction}
    \newtheorem{rem}[lem]{Remark}
    \newtheorem{example}[lem]{Example}

    \newcommand{\mf}[1]{\mathfrak{#1}}
    \newcommand{\mc}[1]{\mathcal{#1}}
    \newcommand{\mb}[1]{\mathbf{#1}}
    \newcommand{\ov}[1]{\overline{#1}}

    \newcommand{\op}{\operatorname}
    \DeclareMathOperator{\Hom}{Hom}
    \DeclareMathOperator{\Spf}{Spf}
    \DeclareMathOperator{\Spec}{Spec}
    \DeclareMathOperator{\Spa}{Spa}
    
    \newcommand{\oc}{\mathrm{oc}}
    \newcommand{\et}{\mathrm{\acute{e}t}}
    \newcommand{\adm}{\mathrm{adm}}
    \newcommand{\proet}{\mathrm{pro\acute{e}t}}
    \newcommand{\dJ}{\mathrm{dJ}}
    \newcommand{\cc}{{\circ\circ}}

    \newcommand{\cat}[1]{\operatorname{\mathbf{#1}}} 
    \newcommand{\Set}{\cat{Set}}
    \newcommand{\Cov}{\cat{Cov}}
    \newcommand{\UCov}{\cat{UCov}}
    \newcommand{\Et}{\cat{\acute{E}t}}
    \newcommand{\FEt}{\cat{F\acute{E}t}}

    \newcommand{\h}{\mathcal{O}}

    \newcommand{\cO}{\mathcal{O}}
    \newcommand*\isomto{%
        \xrightarrow{\raisebox{-0.2 em}{\smash{\ensuremath{\sim}}}}%
    }
    \newcommand{\stacks}[1]{\cite[\href{https://stacks.math.columbia.edu/tag/#1}{Tag #1}]{StacksProject}}
 
    \title{Specialization for the pro-\'etale fundamental group}
    \date{\today}
    
    \author{Piotr Achinger}
    \address{Institute of Mathematics of the Polish Academy of Sciences \newline \indent ul.\ Śniadeckich 8, 00-656 Warsaw, Poland}
    \email{pachinger@impan.pl}

    \author{Marcin Lara}
    \address{Institute of Mathematics of the Polish Academy of Sciences \newline \indent ul.\ Śniadeckich 8, 00-656 Warsaw, Poland}
    \email{marcin.lara@impan.pl}

    \author{Alex Youcis}
    \address{Institute of Mathematics of the Polish Academy of Sciences \newline \indent ul.\ Śniadeckich 8, 00-656 Warsaw, Poland}
    \email{ayoucis@impan.pl}

\begin{document}

\begin{abstract}
    For a formal scheme $\mf{X}$ of finite type over a complete rank one valuation ring, we construct a specialization morphism
    \[
        \pi^\dJ_1(\mf{X}_\eta) \to \pi^\proet_1(\mf{X}_k)
    \]
    from the de Jong fundamental group of the rigid generic fiber to the Bhatt--Scholze pro-\'etale fundamental group of the special fiber. The construction relies on an interplay between admissible blowups of $\mf{X}$ and normalizations of the irreducible components of $\mf{X}_k$, and employs the Berthelot tubes of these irreducible components in an essential way. Using related techniques, we show that under certain smoothness and semistability assumptions, covering spaces in the sense of de Jong of a smooth rigid space which are tame satisfy \'etale descent.
\end{abstract}

\maketitle

\section{Introduction}

Throughout the following we fix a non-archimedean field $K$ with valuation ring $\h_K$, residue field $k$, and a pseudo-uniformizer $\varpi$. Let $X$ be a proper scheme over $\cO_K$ with generic fiber $X_K$ and special fiber $X_k$. In \cite[Exp.\@ X]{SGA1}, Grothendieck constructed a continuous homomorphism
\[ 
	\pi^\et_1(X_K, \ov{x}) \to \pi^\et_1(X_k, \ov{y})
\]
between the \'etale fundamental groups, called the \emph{specialization map}. This map is moreover surjective if $X$ is normal or if $X_k$ is reduced (see \cite[\S 6]{RaynaudPicard} for other criteria). Translated in terms of coverings, the existence of the specialization map amounts to the fact the pullback functor $\FEt_X \to \FEt_{X_k}$ between the corresponding categories of finite \'etale coverings is an equivalence, so that by composing its inverse with the restriction functor $\FEt_{X}\to \FEt_{X_K}$ we obtain a functor
\[
    u\colon \FEt_{X_k} \to \FEt_{X_K}.
\]
Surjectivity of the specialization map then corresponds to the statement that $u$ maps connected coverings to connected coverings.

In cases of geometric interest, such as semistable reduction, the special fiber $X_k$ will often not be normal. For such schemes, the \emph{pro-\'etale fundamental group} introduced by Bhatt and Scholze \cite{BhattScholze} contains more refined information than the \'etale fundamental group. The corresponding notion of a covering space, a \emph{geometric covering}, is an \'etale morphism which satisfies the valuative criterion of properness. Since connected geometric coverings may have infinite degree and may not admit Galois closures, the pro-\'etale fundamental group is not pro-finite (or even pro-discrete) in general, but is a Noohi topological group (see op.\@ cit.\@). As it turns out though, there is no specialization map for the pro-\'etale fundamental group fitting inside a commutative square
\[
	\xymatrix{
		\pi_1^\proet(X_K, \ov{x})\ar[d] \ar@{.>}[r] \ar@{}[r]|-{\displaystyle \times} & \pi_1^\proet(X_k,\ov{y}) \ar[d] \\
		\pi^\et_1(X_K, \ov{x}) \ar[r] & \pi^\et_1(X_k, \ov{y}).
	}
\]

\begin{example}[Tate elliptic curve] \label{ex:tate}
Suppose that $K$ is algebraically closed of characteristic zero and let $X\subseteq \mb{P}^2_{\h_K}$ be a cubic hypersurface whose generic fiber $X_\eta$ is smooth and whose special fiber $X_k$ is nodal. Then $X_k$ is isomorphic to $\mb{P}^1_k$ with $0$ and $\infty$ identified, and has a (unique) geometric covering $Y\to X_k$ with Galois group $\mb{Z}$. Explicitly $Y$ is given by an infinite chain of copies of $\mathbf{P}^1_k$ glued along the poles. In this case, the diagram of fundamental groups as above would have to be of the form
\[ 
    \xymatrix{
        \widehat{\mathbf{Z}}\times \widehat{\mathbf{Z}} \ar@{.>}[rr]^{} \ar[d]_-{\displaystyle \rotatebox{90}{$\sim$}} \ar@{}[rr]|-{\displaystyle \times} & & \mathbf{Z} \ar[d]^{\rm inclusion} \\
        \widehat{\mathbf{Z}}\times \widehat{\mathbf{Z}} \ar[rr]_{\rm projection} & & \widehat{\mathbf{Z}}.
    }
\]
where there is no homomorphism making the square commute.
\end{example}

To clarify the issue, it is useful to note that Grothendieck's specialization map may be understood in terms of a more general specialization map involving formal schemes and rigid spaces. Namely, for any formal scheme $\mf{X}$ locally of finite type over $\h_K$ there is an equivalence $\Et_{\mf{X}_k}\isomto\Et_{\mf{X}}$ and thus one is able to form the functors
\begin{equation*}
    u\colon \Et_{\mf{X}_k}\simeq \Et_{\mf{X}}\to \Et_{\mf{X}_\eta}
    \quad \text{and} \quad
    u\colon \FEt_{\mf{X}_k}\simeq \FEt_{\mf{X}}\to \FEt_{\mf{X}_\eta}
\end{equation*}
where here $\mf{X}_\eta$ is the \emph{rigid generic fiber} in the sense of Raynaud, a rigid $K$-space (see our notation and conventions section for a precise definition). Consequently, we have a formal version of the specialization map $\pi^\et_1(\mf{X}_\eta, \ov{x}) \to \pi^\et_1(\mf{X}_k, \ov{y})$. To obtain Grothendieck's specialization map, one uses Grothendieck's existence theorem to show that if $X$ is a proper $\h_K$-scheme and $\mf{X}=\widehat{X}$ (the $\varpi$-adic completion of $X$) then all finite \'etale covers of $\mf{X}$ and $\mf{X}_\eta=X_K^\mathrm{an}$ are algebraizable (i.e.\@, that $\FEt_{\mf{X}}\simeq \FEt_X$ and $\FEt_{\mf{X}_\eta}\simeq \FEt_{X_K}$). In contrast, in Example~\ref{ex:tate}, if $\mf{Y}\to \mf{X}$ is the unique \'etale lifting of $Y\to X_k$, then the rigid generic fiber $u(Y)=\mf{Y}_\eta\to \mf{X}_\eta$ is the Tate uniformization $\mb{G}_m^{\rm an}\to X_\eta^{\rm an}$, which is patently non-algebraizable.

To accommodate coverings of $\mf{X}_\eta$ arising from geometric coverings of $\mf{X}_k$, one needs a suitably general class of covering spaces in non-archimedean geometry. In the seminal paper \cite{deJongFundamental} de Jong explored a class of morphisms, which we call \emph{de Jong covering spaces}, defined as maps $Y\to X$ of rigid $K$-spaces such that for any point $x$ of $X$ there exists an overconvergent (i.e.\@ Berkovich) open neighborhood $U$ such that $Y_U$ is a disjoint union of finite \'etale coverings of $U$. This class of covering spaces contains finite \'etale coverings as well as more exotic examples, such as Tate's uniformization of elliptic curves and period mappings for certain Rapoport--Zink spaces. In op.\@ cit.\@ it is shown that the category of de Jong covering spaces is sufficiently rich as to support a version of Galois theory, and thus gives rise to a fundamental group which we call the \emph{de Jong fundamental group} and denote $\pi_1^\dJ(X,\ov{x})$.

Our first main result is then the following.

\begin{thm}[See Theorem~\ref{thm:specialization}] \label{thm:intro-spec}
    Let $\mf{X}$ be a quasi-paracompact formal scheme locally of finite type over $\cO_K$, and let $Y\to \mf{X}_k$ be a geometric covering. Then the induced \'etale morphism $u(Y)\to \mf{X}_\eta$ is a de Jong covering space.
\end{thm}

This result is somewhat surprising, since it implies that $u(Y)$ is a disjoint union of finite \'etale coverings on a fairly coarse topology of $\mf{X}_\eta$ despite the fact that $Y$ does not even necessarily split into finite \'etale coverings \'etale locally on $\mf{X}_k$ (see Remark~\ref{rem:etale-nonsplit}). To address this apparent discrepancy, we first establish a criterion for the generic fiber of a morphism of formal schemes to be a de Jong covering space. Namely, for a map of formal schemes $\mf{Y}\to\mf{X}$ such that $\mf{Y}_\eta\to \mf{X}_\eta$ is \'etale, if its restriction to every irreducible component of $\mf{X}_k$ is the disjoint union of finite morphisms, then $\mf{Y}_\eta\to \mf{X}_\eta$ is a de Jong covering space (Proposition~\ref{prop:spec-good-case}). In the second step, we reduce to the situation where this criterion applies using a delicate blowup procedure (see Proposition~\ref{prop:nice-blowup}) and the fact that geometric coverings of a normal scheme are disjoint unions of finite \'etale coverings.

In Corollary~\ref{cor:tfae}, we strengthen the link between de Jong covering spaces and geometric coverings by establishing a converse of Theorem~\ref{thm:intro-spec}. For an \'etale map of formal schemes, its generic fiber is a de Jong covering if and only if its special fiber is a geometric covering, and both are equivalent to the generic fiber being partially proper. 

As a formal consequence of Theorem~\ref{thm:intro-spec} we obtain the specialization homomorphism: assuming that $\mf{X}_\eta$ and $\mf{X}_k$ are both connected, for a compatible choice of geometric points $\ov x$ of $\mf{X}_\eta$ and $\ov y$ of $\mf{X}_k$ (see the discussion preceding the proof of Theorem \ref{thm:specialization}), there is a continuous homomorphism
\begin{equation*} 
    \pi_1^\dJ(\mf{X}_\eta, \ov{x})\to \pi^\proet_1(\mf{X}_k,\ov{y})
\end{equation*}
fitting inside a commutative diagram
\[ 
    \xymatrix{
    	\pi_1^\dJ(\mf{X}_\eta,\ov{x})\ar[d] \ar[r] & \pi_1^\proet(\mf{X}_k,\ov{y}) \ar[d] \\
    	\pi_1^\et(\mf{X}_\eta, \ov{x}) \ar[r] & \pi^\et_1(\mf{X}_k, \ov{y}).
    }
\]
As in the finite \'etale setting, we provide criteria for this homomorphism to have dense image (since the groups are no longer compact, density is not equivalent to surjectivity, and checking the latter could be more challenging). To this end, in Appendix~\ref{app:eta-normal} we define a property of formal schemes $\mf{X}$ over $\h_K$ that we call \emph{$\eta$-normal}. In essence this means that $\mf{X}$ is `integrally closed in $\mf{X}_\eta$', and covers the cases when $\mf{X}$ is normal or $\mf{X}_k$ is reduced. We also show a version of Serre's criterion for $\eta$-normality (see Theorem \ref{prop:serre-crit}). 

We then have the following density result.

\begin{prop}[See Corollary~\ref{cor:specialization}]
    In the situation of Theorem~\ref{thm:intro-spec}, suppose that $\mathfrak{X}$ is $\eta$-normal. Then the specialization map $\pi_1^\dJ(\mf{X}_\eta, \ov{x})\to \pi^\proet_1(\mf{X}_k,\ov{y})$ has dense image.
\end{prop}

The criterion for being a de Jong covering space used in the proof of Theorem~\ref{thm:intro-spec} can be useful even in the presence of ramification. In our second main result, we apply it to study covering spaces which are \emph{tame} in the sense introduced by H\"ubner \cite{Hubner}. We warn the reader that these coverings are only ``tame along the special fiber,'' so to speak (for example, an Artin--Schreier covering of $\mb{A}^1_{\mathbf{F}_p}$ induces a tame covering of the affinoid unit disc over $\mb{Q}_p$). In \cite[\S5.3]{ALY1}, we studied the question whether de Jong coverings of a rigid $K$-space satisfy admissible or \'etale descent, and gave a negative answer by constructing a covering of an annulus in positive characteristic which is not a de Jong covering, but which becomes a de Jong covering space on a non-overconvergent open cover of the base. However, the construction relied on wild ramification phenomena, and it seems natural to expect that such behavior is avoided under a tameness assumption. 

In order to state the result, we need to introduce some terminology used in \cite{ALY1}. Let us call a morphism of rigid $K$-spaces $Y\to X$ an \emph{$\adm$-covering space} (resp.\ an \emph{$\et$-covering space}) if there exists an open (resp.\@ \'etale) cover $\{U_i\to X\}$ such that $Y_{U_i}\to U_i$ is the disjoint union of finite \'etale coverings of $U_i$ for all $i$. 

\begin{thm}[{See Theorem \ref{thm:tame-comparison} and Corollary~\ref{cor:tame-comparison}}]\label{thm:tame-comp-intro}
    Suppose that $K$ is discretely valued. Let $X$ be a smooth quasi-paracompact and quasi-separated rigid $K$-space, and suppose that generalized strictly semistable formal models of $X$ are cofinal (see Definition \ref{def:semistable-cofinal}). Then, every tame $\et$-covering space of $X$ is a de Jong covering.
\end{thm}

In other words, tame de Jong coverings can be glued in the \'etale topology.  The assumption on generalized strictly semistable formal  models is a technical condition that is automatically satisfied if $K$ is of equal characteristic zero (see \cite{TemkinDesingularization}) and conjecturally satisfied in general. Moreover, since tameness is automatic when $K$ has equal characteristic zero, we obtain the following unconditional corollary.

\begin{cor}[See Corollary~\ref{cor:tame-char0}] \label{cor:intro-tame-char-0-cor}
    Let $K$ be a discretely valued non-archimedean field of equal characteristic zero. Then, for any smooth quasi-paracompact and quasi-separated rigid $K$-space $X$, the notions of de Jong, $\adm$-covering spaces, and $\et$-covering spaces coincide.
\end{cor}

As stated above, the proof of Theorem \ref{thm:tame-comp-intro} contains in its kernel the same technique used in the proof of Theorem~\ref{thm:intro-spec}. Suppose for simplicity that $Y\to X$ is a tame \'etale map which splits into the disjoint union of finite \'etale coverings on an open covering of $X$. If $\mf{X}$ is a generalized strictly semistable formal model of $X$ such that this open covering of $X$ extends to an open covering of $\mf{X}$, we show using (a logarithmic version of) Abhyankar's lemma the existence of a formal model $\mf{Y}\to\mf{X}$ which is Kummer \'etale. Using some basic toric geometry we show that $\mf{Y}\to \mf{X}$ satisfies the criterion of Proposition~\ref{prop:spec-good-case}, and deduce that $\mf{Y}_\eta=Y$ is a de Jong covering space of $\mf{X}_\eta=X$ as desired. To extend this result from $\adm$-covering spaces to $\et$-covering spaces we apply an \'etale bootstrap argument in the sense of \cite[\S2.8]{ALY1} (see Corollary \ref{cor:tame-comparison}).

It seems believable that the smoothness and semistability assumptions in Theorem~\ref{thm:tame-comp-intro} are not necessary, but the authors were unable to remove them. In a different direction, the link with logarithmic geometry used in the proof strongly suggests the existence of a specialization map to a logarithmic variant of the pro-\'etale fundamental group of the special fiber of a semistable model, and we would expect such a map to be an isomorphism. Finally, in the companion paper \cite{ALY1}, the authors have developed the notion of a 'geometric covering' of a rigid space which provides a generalization of $\et$-covering spaces. A natural question is whether in the setting of Corollary \ref{cor:intro-tame-char-0-cor} every such tame geometric covering must be a de Jong covering space.

\subsection*{Notation and Conventions} Throughout this paper, we shall use the following notation and terminology:
\begin{itemize}
    \item By a \emph{non-archimedean field} $K$ we mean a field complete with respect to a rank one valuation $|\cdot|\colon K\to [0, \infty)$.
    \item For a Huber ring $A$ we denote by $A^\circ$ the subring of powerbounded elements and by $A^\cc$ the set of topologically nilpotent elements. We abbreviate $\Spa(A,A^\circ)$ to $\Spa(A)$.
    \item By a \emph{rigid $K$-space} we mean an adic space locally of finite type over $\Spa(K)$. Our conventions and notation concerning adic spaces is as in \cite[\S 1-2]{ALY1}. We denote by $\cat{Rig}_K$ the category of rigid $K$-spaces, and by $\cat{Rig}^{\rm qcqs}_K$ the full subcategory consisting of quasi-compact and quasi-separated rigid $K$-spaces.
    \item A topological space $X$ is \emph{quasi-paracompact} \cite[Definition~8.2/12]{BoschLectures} if it admits an open cover $X=\bigcup U_i$ by quasi-compact opens such that each $U_i$ intersects only finitely many other $U_j$.
    \item An open subset $U$ of a rigid $K$-space $X$ is \emph{overconvergent} if it is stable under specialization, or equivalently if the inclusion $U\hookrightarrow X$ is partially proper \cite[Proposition~2.3.4]{ALY1}.
    \item By a \emph{maximal point} of a rigid $K$-space $X$ we mean a point which is maximal with respect to the generalization relation. By \cite[Lemma 1.1.10]{Huberbook} this is equivalent to the corresponding valuation being rank $1$.
    \item Let $X$ be an adic space, formal scheme, or scheme. By $\Et_X$ (resp.\ $\FEt_X$) we mean the category of objects \'etale (resp.\ finite \'etale) over $X$.
    \item Closed subsets of (formal) schemes are implicitly treated as schemes endowed with the reduced scheme structure.
    \item We consistently use the term \emph{cover} as in `open cover' and \emph{covering} as in `covering space.'
\end{itemize}

\subsection*{Acknowledgements}

The authors would like to thank Hélène Esnault, Joachim Jelisiejew and Michael Temkin for helpful conversations. 

This work is a part of the project KAPIBARA supported by the funding from the European Research Council (ERC) under the European Union’s Horizon 2020 research and innovation programme (grant agreement No 802787).

\section{Rigid generic fibers of formal schemes}
\label{s:prelims}

In this section, we review the theory of formal schemes and their rigid generic fibers.  As before, we fix a non-archimedean field $K$ with ring of integers $\h_K$, residue field $k$, and a pseudo-uniformizer $\varpi$. We will need to work with certain non-adic formal schemes over $\cO_K$, such as $\Spf(\cO_K\llbracket x \rrbracket)$; for such formal schemes the rigid generic fiber has been constructed by Berthelot \cite{Berthelot} (see also de Jong \cite[\S 7]{deJongCrystal}) in case $K$ is discretely valued. For an extension to $K$ arbitrary, we will use the approach of Fujiwara and Kato \cite[Chapter II, Section 9.6]{FujiwaraKato} (see also Scholze and Weinstein \cite{ScholzeWeinstein} for an alternative definition). 

\subsection{Formal schemes}
\label{subsec:formal-schemes}

We refer to \cite[Chapter I]{FujiwaraKato} for the basic terminology regarding formal schemes. We shall only consider formal schemes which are locally of the form $\Spf(A)$ where $A$ has the $I$-adic topology for a finitely generated ideal $I$, and is $I$-adically complete and separated. By an ideal of definition of $A$ we shall always mean a finitely generated ideal of definition. Recall that a continuous homomorphism of such algebras $A\to B$ is \emph{adic} if $I\cdot B$ is an ideal of definition of $B$ for some (equiv.\ every) ideal of definition $I$ of $A$, and a map of formal schemes $\mf{Y}\to \mf{X}$ is \emph{adic} if it is locally of the form $\Spf(B)\to \Spf(A)$ for an adic homomorphism $A\to B$.

We first discuss a hypothesis on formal schemes over $\h_K$ that is more flexible than locally topologically of finite type (in the discrete case such objects were considered in \cite[\S1]{BerkovichVanishingII} under the terminology `special').

\begin{prop} \label{prop:fml-lfft}
    Let $A$ be a topological $\cO_K$-algebra which is complete and separated with respect to a finitely generated ideal. The following conditions are then equivalent:
    \begin{enumerate}[(a)]
        \item For every (equiv.\@ any) pseudouniformizer $\varpi$ of $\cO_K$ and ideal of definition $I\subseteq A$ such that $\varpi A\subseteq I$, the ring $A/I$ is a finitely generated $\cO_K/\varpi$-algebra.
        \item There is a continuous and adic $\cO_K$-algebra surjection
        \[ 
            \psi\colon P:= \cO_K\langle x_1, \ldots, x_n\rangle \llbracket y_1, \ldots, y_m\rrbracket \to A
        \]
        for some $n, m\geqslant 0$. (Here, $P$ is the completion of $\cO_K[x_1, \ldots, x_n, y_1, \ldots, y_m]$ with respect to the ideal\footnote{N.B.\ We have $\cO_K\langle x_1, \ldots, x_n\rangle \llbracket y_1, \ldots, y_m\rrbracket \simeq \cO_K \llbracket y_1, \ldots, y_m\rrbracket\langle x_1, \ldots, x_n\rangle$ where the latter is the ring of restricted power series as in \cite[Chapter 0, \S8.4]{FujiwaraKato}.} $(\varpi, y_1, \ldots, y_m)$.)
    \end{enumerate}
\end{prop} 

\begin{proof}
Suppose that $A/I$ is finitely generated over $\cO_K/\varpi$ and choose a surjection 
\[ 
    \varphi\colon (\cO_K/\varpi)[x_1, \ldots, x_n]\to A/I. 
\]
Write $I = (f_1, \ldots, f_m, \varpi)$, and let $\psi \colon P \to A$ be the unique continuous $\cO_K$-algebra homomorphism sending $x_i$ to prescribed lifts of $\varphi(x_i)$ and with $\psi(y_i)=f_i$. Such a homomorphism exists because $f_i$ are topologically nilpotent. Since $P$ is by definition the $(y_1, \ldots, y_m, \varpi)$-adic completion of $\cO_K[x_1, \ldots, x_n, y_1, \ldots, y_m]$, the map $\psi$ is adic. Therefore by \cite[Chapter I, Proposition 4.3.6]{FujiwaraKato} the map $\psi$ is surjective, and hence condition (b) is satisfied. 

For the other direction, note that there exists a $k\geqslant 1$ such that $(\psi(y_1^k), \ldots, \psi(y_m^k), \varpi) \subseteq I$. Then $P/(\psi(y_1^k), \ldots, \psi(y_m^k), \varpi)\simeq (\cO_K/\varpi)[x_1,\ldots,x_n,y_1,\ldots,y_m]/(y_1^k,\ldots,y_m^k)$ is a finitely generated $\cO_K/\varpi$-algebra, and hence so is its quotient $A/I$.
\end{proof}

\begin{definition}[{see \cite[Chapter II, Definition 9.6.5]{FujiwaraKato}}] \label{def:fsch}
    Let $A$ be a topological $\cO_K$-algebra which is complete and separated with respect to a finitely generated ideal.
    \begin{enumerate}[(a)] 
        \item We say that $A$ is \emph{topologically formally of finite type} if it satisfies any of the equivalent conditions of Proposition~\ref{prop:fml-lfft}, \emph{topologically of finite type} if it is formally topologically of finite type and adic over $\cO_K$ (equiv.\ an adic quotient of $\cO_K\langle x_1, \ldots, x_n\rangle$ for some $n\geqslant 0$), and \emph{admissible} if it is topologically of finite type and flat (equiv.\ torsion free) over $\cO_K$.
        \item We say that a formal scheme $\mf{X}$ over $\h_K$ is \emph{locally formally of finite type} (resp.\ \emph{locally of finite type}, resp.\ \emph{admissible}) if for all affine open covers $\{\Spf(A_i)\}$ (equiv.\ for a single such open cover) each $A_i$ is topologically formally of finite type (resp.\ topologically of finite type, resp.\ admissible).
        \item We say that a formal scheme $\mf{X}$ over $\h_K$ is \emph{(formally) of finite type} if it is locally (formally) of finite type and quasi-compact.
    \end{enumerate}
\end{definition}

We shall denote the category of formal schemes locally formally of finite type (resp.\ locally of finite type, resp.\ of finite type) over $\cO_K$ by $\cat{F Sch}_{\cO_K}^{\rm lfft}$ (resp.\ $\cat{FSch}_{\cO_K}^{\rm lft}$, resp.\ $\cat{FSch}_{\cO_K}^{\rm ft}$). The notions of (locally) formally of finite type, (locally) of finite type, and admissible are preserved under base change along $\h_K\to\h_{K'}$ for $K'$ a non-archimedean extension of $K$.

\medskip\paragraph{Admissible ideals}

For a formal scheme $\mf{X}$ we say that an ideal sheaf $\mc{I}\subseteq \h_\mf{X}$ is an \emph{ideal sheaf of definition} if it is adically quasi-coherent (see \cite[Chapter I, Definition 3.1.3]{FujiwaraKato}) and if for all affine open formal subschemes $U$ we have that $\mc{I}(U)$ is an ideal of definition of $\cO_{\mf{X}}(U)$. If $\mf{X}$ is quasi-compact and quasi-separated then the set of ideal sheaves of definition is cofiltering, and in particular non-empty (cf.\@ \cite[Chapter I, Corollary 3.7.12]{FujiwaraKato}). We say that an ideal sheaf $\mc{J}\subseteq \h_\mf{X}$ is \emph{admissible} if it is adically quasi-coherent, of finite type, and open (i.e.\@ locally contains an ideal of definition). For an admissible ideal sheaf $\mc{J}$ we denote by $\mf{X}(\mc{J})$ the scheme $(|\mf{X}|,\h_\mf{X}/\mc{J})$.

\medskip\paragraph{Admissible blowups}

Let $\mf{X}$ be a formal scheme and $\mc{J}$ an admissible ideal sheaf of $\mf{X}$. There then exists a final object amongst morphisms $\pi\colon \mf{X}'\to\mf{X}$ which are adic, proper, and satisfy $\mc{J}\h_{\mf{X}'}$ is invertible. We call this universal object the \emph{admissible blowup} of $\mf{X}$ relative to $\mc{J}$ and denote it by $\pi_\mc{J}\colon \mf{X}_\mc{J}\to \mf{X}$. We call the subscheme cut out by $\mc{J}$ the \emph{center} of the admissible blowup $\mf{X}'\to\mf{X}$. One can give an explicit description of this admissible blowup as in classical algebraic geometry (cf.\@ \cite[Chapter II, \S1.1.(a) and \S1.1.(b)]{FujiwaraKato}). Since admissible blowups are finite type and do not introduce torsion for an ideal of definition (see \cite[Chapter II, Corollary 1.1.6]{FujiwaraKato} for this latter claim), one sees that the properties defined in Definition~\ref{def:fsch}(bc) are stable under admissible blowups.

\medskip\paragraph{Strict transform}

If $\mf{X}_\mc{J}\to \mf{X}$ is an admissible blowup and $\mf{Y}\to \mf{X}$ an adic morphism, then $\mf{Y}_{\mf{X}_\mc{J}}\to \mf{Y}$ will, in general, not be an admissible blowup. That said, one can show (see \cite[Chapter II, \S1.2]{FujiwaraKato}) that the $\mc{J}$-torsion ideal sheaf  $\mc{K}=\h_{\mf{Y}_{\mf{X}_\mc{J}},\mc{J}\text{-tors}}$ is of finite type. One then considers the closed formal subscheme $\mf{Y}'$ of $\mf{Y}_{\mf{X}_{\mc{J}}}$ cut out by $\mc{K}$ which is called the \emph{strict transform} of $\mf{Y}\to \mf{X}$ along $\mf{X}_\mc{J}\to\mf{X}$. The map $\mf{Y}'\to \mf{Y}$ is isomorphic to the admissible blowup $\mf{Y}_{\mc{J}\h_{\mf{Y}}}$ (see \cite[Chapter II, Proposition 1.2.9]{FujiwaraKato}).

\medskip\paragraph{Underlying reduced subscheme}

For a formal scheme $\mf{X}$ there exists a unique ideal sheaf $\h_\mf{X}^\cc\subseteq \h_\mf{X}$ such that $\h_\mf{X}^\cc(\Spf(A))=A^\cc$ for all affine open $\Spf(A)\subseteq \mf{X}$. We call $\h_\mf{X}^\cc$ the \emph{ideal sheaf of topologically nilpotent elements}. The pair $\underline{\mf{X}}:=(\mf{X},\h_\mf{X}/\h_\mf{X}^\cc)$ defines a closed subscheme of $\mf{X}$ called the \emph{underlying reduced subscheme} of $\mf{X}$. For any ideal sheaf of definition $\mc{I}$ of $\mf{X}$ the map $\underline{\mf{X}}\to \mf{X}$ factorizes through $\mf{X}(\mc{I})$ and identifies $\underline{\mf{X}}$ with the underlying reduced subscheme of $\mf{X}(\mc{I})$.

\medskip\paragraph{Completions}

Let $\mf{X}$ be a formal scheme and let $\mc{J}\subseteq \h_\mf{X}$ be an admissible ideal sheaf. Consider the colimit $\mf{X}(\mc{J}^\infty)=\varinjlim \mf{X}(\mc{J}^n)$\footnote{Here we consider $\mf{X}(\mc{J}^n)$ as a \emph{pseudo-discrete formal scheme} as in \stacks{0AHY}.}\label{footnote:pd} in the category of topologically locally ringed spaces. Let us note that for all $n\geqslant 1$ the underlying topological space $\mf{X}(\mc{J}^n)$ is constant. Let us call it $X$. We then consider the sheaves $\h_{\mf{X}(\mc{J}^n)}$ on $X$ for all $n$. We have obvious surjections $\h_{\mf{X}(\mc{J}^i)}\to \h_{\mf{X}(\mc{J}^j)}$ for $i\geqslant j\geqslant 1$. Using \cite[Chapter I, Proposition 1.4.2]{FujiwaraKato} one shows that $\mf{X}(\mc{J}^\infty)$ is a formal scheme. If $\mf{X}=\Spf(A)$ and $J\subseteq A$ is the ideal corresponding to $\mc{J}$, then $\mf{X}(\mc{J}^\infty)$ is in fact equal to $\Spf(B)$ with $B$ the $J$-adic completion of $A$. 

The proof of the following representability result is left to the reader.

\begin{prop} \label{prop:completion-functor-description-2} 
    Let $\mf{X}$ be a formal scheme such that $\underline{\mf{X}}$ is locally Noetherian. Then, for any closed subset $Z$ of $|\mf{X}|$ the functor
    \begin{equation*}
        \widehat{\mf{X}}_{Z}\colon \cat{FSch}^{\rm op}\to\Set,\qquad \mf{X}'\mapsto \left\{f\in\Hom_{\cat{FSch}}(\mf{X}',\mf{X})\,:\,
        f(|\mf{X}'|)\subseteq Z \right\} 
    \end{equation*}
    is representable. Moreover, if $|\mf{X}(\mc{J})|=Z$ for an admissible ideal sheaf $\mc{J}\subseteq \h_\mf{X}$ then it is represented by $\mf{X}(\mc{J}^\infty)$. 
\end{prop}

We call the formal scheme  $\widehat{\mf{X}}_Z$ the \emph{completion of $\mf{X}$ along $Z$}. It is clear that if $\mf{Y}$ is an open formal subscheme of $\mf{X}$ then $\widehat{\mf{U}}_{U\cap Z}$ is the preimage of $\mf{U}$ in $\widehat{\mf{X}}_Z$. By reducing to the affine case using \cite[Chapter I, Proposition 3.7.11]{FujiwaraKato} one can show that if $\mf{X}$ is moreover quasi-compact and quasi-separated, then every closed subset $Z$ is of the form $|\mf{X}(\mc{J})|$ for some admissible ideal sheaf $\mc{J}\subseteq \h_{\mf{X}}$.

The following result will play an important role in the proof of Proposition \ref{prop:spec-good-case}. In its statement, we use the notion of \emph{adically locally of finite presentation} as in \cite[Chapter I, Definition 5.3.1]{FujiwaraKato}. Note that by \cite[Chapter I, Corollary 2.2.4]{FujiwaraKato} and \cite[Chapter 0, Corollary 9.2.9]{FujiwaraKato} this condition is automatic for a morphism of admissible formal schemes over $\h_K$.

\begin{prop} \label{prop:alt-refined-topological-invariance}
    Let $\mf{Y}\to\mf{X}$ be a morphism adically locally of finite presentation of formal schemes where $\underline{\mf{X}}$ is locally Noetherian. Let $Z$ be a closed subset of $\mf{X}$ and $\mc{Z}$ any closed subscheme of $\mf{X}$ with $|\mc{Z}|=Z$. Then if $\mf{Y}_\mc{Z}\to \mc{Z}$ is finite (resp.\@ a disjoint union of finite morphisms) then, the same is true for $\mf{Y}_{\widehat{\mf{X}}_Z}\to\widehat{\mf{X}}_Z$.
\end{prop}

\begin{proof} 
Suppose first that the map is finite. Without loss of generality, we may assume that $\mc{Z}=Z$ with the reduced scheme structure. By working locally on the target, we may assume that $\mf{X}=\Spf(A)$ and that $Z_\mathrm{red}=\Spec(A/J)_\mathrm{red}$ for some admissible ideal $J$ of $A$. Set $T=\Spec(A/J)$. Note then that $\mf{Y}_T\to T$ is, by assumption, a morphism of schemes locally of finite presentation. By assumption, we have that the pullback of this map along $T_\mathrm{red}\to T$ is finite, and thus it is itself finite by \stacks{0BPG}. Note though that since $J^\Delta$ is an ideal sheaf of definition of $\widehat{\mf{X}}_Z$ that by \cite[Chapter I, Proposition 4.2.3]{FujiwaraKato} to verify $\mf{Y}_{\widehat{\mf{X}}_Z}\to\widehat{\mf{X}}_Z$ is finite, it suffices to check this after base change along $T\to\widehat{\mf{X}}_Z$ from where the claim follows. The second claim is clear by the above discussion since $\mf{Y}_{\widehat{\mf{X}}_{Z}}\to \mf{Y}_{Z_\mathrm{red}}$ is a homeomorphism, and so if $\mf{Y}_{Z_\mathrm{red}}$ is a disjoint union of clopen subsets finite over $Z_\mathrm{red}$, then $\mf{Y}_{\widehat{\mf{X}}_Z}$ is a disjoint union of clopen subsets which, by the above, must be finite over $\widehat{\mf{X}}_Z$.
\end{proof}

\subsection{Generic fibers of formal schemes}
\label{ss:generic-fiber} 

In this subsection we discuss the notion of the \emph{rigid generic fiber} of a formal scheme locally formally of finite type over $\cO_K$, which is a rigid $K$-space. We then note several properties of the generic fiber construction.

We first recall Raynaud's equivalence as developed by Fujiwara--Kato in \cite{FujiwaraKato}, which in particular gives the construction for formal schemes locally of finite type over $\cO_K$. Let $A$ be a topologically finite type $\cO_K$-algebra. The algebra $A_K = A[\frac 1 \pi]$ is then an affinoid $K$-algebra, i.e.\ an algebra topologically of finite type over $K$. The subring $A^\circ_K\subseteq A_K$ of powerbounded elements coincides with the integral closure of (the image of) $A$ in $A_K$ (see \cite[Chapter II, Corollary A.4.10]{FujiwaraKato}). Combining \cite[Chapter II]{FujiwaraKato} and \cite[Chapter II, Theorem A.5.1]{FujiwaraKato}, there exists a unique functor
\[
    (-)_\eta \colon \cat{FSch}^{\rm ft}_{\cO_K} \to \cat{Rig}_K^{\rm qcqs}
\]
such that $\Spf(A)_\eta = \Spa(A_K)$ for every topologically finite type $\cO_K$-algebra $A$, and which respects open immersions and open covers. This functor naturally extends to a functor 
\[
    (-)_\eta \colon \cat{FSch}^{\rm lft}_{\cO_K} \to \cat{Rig}_K^{\rm qs}, 
\]
and for $\mf{X}$ locally of finite type over $\cO_K$, the rigid $K$-space $\mf{X}_\eta$ is called the \emph{rigid generic fiber} of $\mf{X}$. Furthermore, $(-)_\eta$ sends the class $W$ of admissible blowups to isomorphisms and induces equivalences of categories
\begin{equation}\label{eq:Raynaud-equiv}
    \cat{FSch}^{\rm adm, qc}_{\cO_K}[W^{-1}] \underset{\rm incl}\isomto \cat{FSch}^{\rm ft}_{\cO_K}[W^{-1}] \underset{(-)_\eta}\isomto \cat{Rig}_K^{\rm qcqs}.
\end{equation}
Here $\cat{FSch}^{\rm adm, qc}_{\cO_K}$ consists of quasi-compact admissible formal schemes over $\h_K$, and $(-)[W^{-1}]$ denotes the localization with respect to $W$, i.e.\ the category obtained by formally inverting all admissible blowups. By a \emph{formal model} of a rigid $K$-space $X$ we shall mean a formal scheme $\mf{X}$ locally of finite type over $\cO_K$ together with an isomorphism $\mf{X}_\eta\simeq X$. By \cite[Chapter II, Proposition 2.1.10]{FujiwaraKato}, for any quasi-compact and quasi-separated rigid $K$-space, the category of \emph{admissible} formal models of $X$ is (equivalent to) a cofiltering poset.

If $X$ is a quasi-compact and quasi-separated rigid $K$-space, then the construction of the rigid generic fiber allows one to identify the locally topologically ringed space $(X,\h_X^+)$ as $\varprojlim\, (\mf{X},\h_\mf{X})$ where $\mf{X}$ runs over admissible formal models of $X$. In particular, for each admissible formal model $\mf{X}$ of $X$ we have a morphism of locally topologically ringed spaces
\[ 
    \op{sp}_\mf{X} \colon (X,\h_X^+)\to (\mf{X},\h_\mf{X}).
\]
called the \emph{specialization map} for $\mf{X}$. For $\mf{X}=\Spf(A)$ affine, the underlying map of topological spaces sends a valuation $\nu\colon A_K\to \Gamma\cup\{0\}$ to the (open) prime ideal $\{x\in A\,:\, \nu(x)<1\}$ and the map on global sections is the natural map $\h_\mf{X}(\mf{X})=A\to A_K^\circ =\h_X(X)^+$.

Since $|\mf{X}|=|\mf{X}_k|$ we shall often implicitly treat $\op{sp}_\mf{X}$ as a map of topological spaces $|X|\to |\mf{X}_k|$. This map is continuous, quasi-compact, closed, and surjective (see \cite[Chapter II, Theorem 3.1.2]{FujiwaraKato} and \cite[Chapter II, Proposition 3.1.5]{FujiwaraKato}). If $\varphi\colon \mf{X}'\to \mf{X}$ is a morphism in $\cat{FSch}_{\h_K}^\mathrm{ft}$ then the diagram
\begin{equation*}
    \xymatrix{(X',\h_{X'}^+)\ar[d]_{\varphi_\eta}\ar[r]^{\op{sp}_{\mf{X}'}} & (\mf{X}',\h_{\mf{X}'})\ar[d]^\varphi\\ (X,\h_X^+)\ar[r]_{\op{sp}_\mf{X}} & (\mf{X},\h_\mf{X})\\ }
\end{equation*}
commutes. Moreover, for any open subset $\mf{U}$ of $\mf{X}$ the induced map $\mf{U}_\eta\to \mf{X}_\eta$ is an open immersion with image $\op{sp}_\mf{X}^{-1}(\mf{U})$ (see \cite[Chapter II, Proposition 3.1.3]{FujiwaraKato}). It follows that the definition of the specialization map may be extended to formal schemes $\mf{X}$ locally of finite type over $\h_K$, and that this extension enjoys also the property of being continuous, quasi-compact, closed and surjective if $\mf{X}$ is admissible.

In order to extend the definition of the rigid generic fiber to the category $\cat{FSch}^{\rm lfft}_{\cO_K}$, we need the following result. 
\begin{prop}[{\cite[Chapter II, \S9.6.(b)]{FujiwaraKato}}]\label{lem:generic-fiber-functor-of-points-1}
    Let $\mf{X}$ be a formal scheme locally formally of finite type over $\cO_K$. Then, the functor
    \begin{equation*}
        \left(\cat{Rig}_K^\mathrm{qcqs}\right)^{\rm op}\to \cat{Set},\qquad Z\mapsto \varinjlim\Hom_{\cO_K}(\mf{Z},\mf{X})
    \end{equation*}
    where $\mf{Z}$ runs over the admissible formal models of $Z$, is representable by a rigid $K$-space $\mf{X}_\eta$.  
\end{prop}

We call $\mf{X}_\eta$ the \emph{rigid generic fiber} of $\mf{X}$. If $\mf{X}$ is locally of finite type over $\cO_K$, this agrees with the previous definition of $\mf{X}_\eta$. The association $\mf{X}\mapsto \mf{X}_\eta$ is functorial in $\mf{X}$ and commutes with finite limits and disjoint unions. Moreover, it sends admissible blowups to isomorphisms. 

The construction below makes the rigid generic fiber more explicit in the affine case. 

\medskip

\begin{construction}[The rigid generic fiber of an affine formal scheme] \label{const:generic-fiber-affine}
Let $\mf{X}=\Spf(B)$ for a topologically formally of finite type $\cO_K$-algebra $B$, and let $J=(b_1,\ldots,b_m)$ be an ideal of definition of $B$. We denote by $B[\frac{J}{\varpi}]$ the affine blowup algebra \stacks{052Q}, i.e.\ the image of $B[x_1, \ldots, x_m]/(\pi x_i - b_i)$ in $B[\frac 1 \varpi]$; it is independent of the choice of generators of $J$. Let $B\langle\frac{J}{\varpi}\rangle$ be the $J$-adic completion of $B[\frac{J}\varpi]$, which is an admissible $\cO_K$-algebra on which the $J$-adic topology coincides with the $\varpi$-adic topology. The map $B\to B\langle \frac J \varpi \rangle$ is continuous, and the morphism $\Spf(B\langle\frac{J}{\varpi}\rangle)\to\Spf(B)$ induces a map of rigid $K$-spaces $\Spf(B\langle \frac J \varpi \rangle)_\eta\to \Spf(B)_\eta$. Writing $B(J) = B\langle\frac{J}{\varpi}\rangle [\frac 1 \varpi]$, we have $\Spf(B\langle \frac J \varpi \rangle)_\eta = \Spa(B(J))$.\footnote{Note that $\Spa(B(J))$ may also be described as the adic spectrum of a fairly simple non-complete Huber pair. Namely, $\Spa(B(J))$ is equal to $\Spa(B[\frac{1}{\varpi}],B[\frac{J}{\varpi}]^\sim)$ where $B[\frac{1}{\varpi}]$ has the unique structure of a Huber ring so that $B[\frac{J}{\varpi}]$ equipped with the $J$-adic topology is an open subring, and where $B[\frac{J}{\varpi}]^\sim$ denotes the integral closure of $B[\frac{J}{\varpi}]$ in $B[\frac{1}{\varpi}]$.}

For $J'\subseteq J$, we have an inclusion $B[\frac{J'}{\varpi}]\subseteq B[\frac{J}{\varpi}]$ and consequently a morphism $\Spf(B[\frac{J}{\varpi}])\to\Spf(B[\frac{J'}{\varpi}])$ over $\Spf(B)$. The induced morphism $\Spa(B(J))\to \Spa(B(J'))$ is an isomorphism onto a rational open domain of $\Spa(B(J'))$. Thus, the inductive system $\{\Spa(B(J))\}$  indexed by all finitely generated ideals of definition of $B$ gives a well-defined adic space $\varinjlim_J \Spa(B(J))$. Since this system admits compatible maps to $\mf{X}_\eta$, we obtain a morphism of rigid $K$-spaces
\begin{equation} \label{eqn:limit-BJ}
    \varinjlim_J \Spa(B(J))\to \mf{X}_\eta.
\end{equation}
\end{construction}

\medskip

One then has the following concrete description of the generic fiber $\mf{X}_\eta$.

\begin{lem}[{\cite[Chapter II, Remark 9.6.3]{FujiwaraKato}}]\label{lem:concrete-generic-fiber-description}
     With notation as in Construction \ref{const:generic-fiber-affine} the map \eqref{eqn:limit-BJ} is an isomorphism.
\end{lem}

Let us note that for any given ideal of definition $J_0$ of $B$ that the set $\{J_0^n\}$ is cofinal in the set of all ideals of definition of $B$. In particular, in Lemma \ref{lem:concrete-generic-fiber-description} we may replace $\varinjlim_J \Spa(B(J))$ with $\varinjlim_n \Spa(B(J_0^n))$. We will often do this without comment below.

With this, we can give a more concrete description of the functor of points of the generic fiber in the situation dictated in Construction \ref{const:generic-fiber-affine}. It will enable us to show easily that an \'etale morphism of formal schemes induces an \'etale morphism on rigid generic fibers.

\begin{construction} \label{const:generic-fiber-affine-functor-of-points}
Let $\mf{X}=\Spf(B)$ and $\mf{X}'=\Spf(B')$ be affine objects of $\cat{FSch}^\mathrm{lfft}_{\h_K}$ and let $\mf{X}'\to\mf{X}$ be a morphism over $\h_K$. Let $J$ be an ideal of definition of $B$ and let $J'$ be an ideal of deifnition of $B'$ such that $JB\subseteq J'$ (which exists since $B\to B'$ is continuous). It is then easy to see that under the map $\mf{X}'_\eta\to\mf{X}_\eta$ that $\Spa(B((J')^n))$ maps into $\Spa(B(J^n))$ for all $n$.

Let $(R,R^+)$ be a Huber $(K,\h_K)$-algebra. Then, any morphism of adic spaces $\Spa(R,R^+)\to \mf{X}_\eta$ over $(K,\h_K)$ must factorize through $\Spa(B(J^n))$ for some $n$, and thus defines a map of Huber pairs $(B(J^n),B(J^n)^\circ)\to (R,R^+)$ which, in turn, defines a continuous map $B\to R^+$ of $\h_K$-algebras independent of the choice of $n$. Moreover, this map must send $\varpi$ to an element of $R^\times$. Note that one has a natural map
\begin{equation*}
    \Hom_{\mf{X}_\eta}(\Spa(R,R^+),\mf{X}'_\eta)\to \Hom_B(B',R^+).
\end{equation*}
Indeed, any map $\Spa(R,R^+)\to \mf{X}'_\eta$ must factorize through $\Spa(B((J')^m))$ for some $m\geqslant n$ and so defines a map of Huber pairs $(B((J')^m),B((J')^m)^\circ)\to (R,R^+)$. This then defines a continuous map $B'\to R^+$ of $B$-algebras which is independent of $m$.
\end{construction} 

\begin{lem} \label{lem:generic-fiber-functor-of-points-2}
    With notation as in Construction \ref{const:generic-fiber-affine-functor-of-points}, the map
    \begin{equation*}
        \Hom_{\mf{X}_\eta}(\Spa(R,R^+),\mf{X}'_\eta)\to\Hom_B(B',R^+)
    \end{equation*}
    is a functorial bijection. Moreover,
    \begin{equation*}
        \Hom_B(B',R^+)=\varinjlim_{R_0}\Hom_{\mf{X}}(\Spf(R_0),\mf{X}')
    \end{equation*}
    as $R_0$ runs over the open and bounded $B$-subalgebras of $R^+$.
\end{lem}

With this, we can prove the following.

\begin{prop}\label{prop:generic-fiber-props}
    Let $f\colon \mf{X}'\to\mf{X}$ be an \'etale (resp.\@ finite) morphism in $\cat{FSch}^\mathrm{lfft}_{\h_K}$. Then, $f_\eta\colon \mf{X}'_\eta\to\mf{X}_\eta$ is \'etale (resp.\@ finite).
\end{prop}

\begin{proof}
We may assume that $\mf{X}=\Spa(B)$ and $\mf{X}'=\Spa(B')$. Suppose first that $\mf{X}'\to\mf{X}$ is \'etale. We deduce from Lemma \ref{lem:generic-fiber-functor-of-points-2} that for any Huber $(K,\h_K)$-algebra $(R,R^+)$ and square-zero ideal $N$ of $R$ one has a commutative square
\begin{equation*}
    \xymatrix{\Hom_{\mf{X}_\eta}(\Spa(R,R^+),\mf{X}'_\eta)\ar[r]\ar[d] & \varinjlim_{R_0}\Hom_{\mf{X}}(\Spf(R_0),\mf{X}')\ar[d]\\ \Hom_{\mf{X}_\eta}(\Spa(R/N,R^+/(R^+\cap N)),\mf{X}'_\eta)\ar[r] & \varinjlim_{R_0}\Hom_{\mf{X}}(\Spf(R_0/N),\mf{X}')}
\end{equation*}
where the horizontal maps are bijections. But, the right vertical arrow is a bijection since $\mf{X}'\to\mf{X}$ is \'etale, and thus the left vertical arrow must also be a bijection. Since $f_\eta\colon \mf{X}'_\eta\to\mf{X}_\eta$ is locally of finite type (since both are locally of finite type over $\Spa(K)$) we deduce that $f_\eta$ is \'etale as desired.

Now suppose that $f$ is finite, and let $J$ be an ideal of definition of $B$. Since $f$ is adic, the ideal $J'=J B'$ is an ideal of definition of $B'$. As $B \to B'$ is finite, $B'\langle\frac{(J')^n}{\varpi}\rangle = B\langle\frac{J^n}{\varpi}\rangle \otimes_B B'$ and it follows that $B'((J')^n) = B(J^n)\otimes_B B'$. Since $\mf{X}'_\eta\to\mf{X}_\eta$ is the colimit of the maps $\Spa(B'((J')^n)) \to \Spa(B(J^n))$, which are finite by \cite[Lemma 1.4.5 iv)]{Huberbook}, it follows that this map is finite, as desired.
\end{proof}

We end this subsection by showing that the generic fiber preserves the partial properness of the special fiber.\footnote{While we do not need the converse, we still expect this implication to hold.} Recall here that a morphism of schemes is called \emph{partially proper} if it satisfies the valuative criterion for properness (see \stacks{03IX}). A morphism of rigid $K$-spaces is called \emph{partially proper} if it is separated and universally specializing (see \cite[1.3.3]{Huberbook}). 

\begin{lem}\label{lem:pp-generic-special}
    Let $\mf{Y}\to\mf{X}$ be a morphism of admissible formal schemes over $\h_K$, and assume that $\mf{X}$ is quasi-paracompact. If $\mf{Y}_\eta\to \mf{X}_\eta$ is partially proper then so is $\mf{Y}_k\to\mf{X}_k$.
\end{lem}

Before proving the lemma, we recall some facts about taut spaces. A locally spectral topological space $X$ is \emph{taut} if it is quasi-separated and the closure of every quasi-compact open is quasi-compact \cite[Definition~5.1.2]{Huberbook}. It is easy to check that a locally topologically Noetherian space is taut if and only if its irreducible components are quasi-compact. Moreover, if $X$ is quasi-paracompact and quasi-separated then $X$ is taut by \cite[Lemma~5.1.3(ii)]{Huberbook}. If $Y\to X$ is a partially proper morphism of rigid $K$-spaces and $X$ is taut, then so is $Y$, by \cite[Lemma 5.1.4]{Huberbook}. Finally, for an admissible formal scheme $\mf{Y}$ over $\cO_K$, since specialization map $\op{sp}_{\mf{Y}}$ is closed, quasi-compact, and surjective, we deduce that $\mf{Y}_\eta$ is taut if and only if $\mf{Y}_k$ is taut. 

\begin{lem}\label{lem:universally-spec}
    Let $f\colon Y\to X$ be a separated morphism of schemes locally of finite type over $k$ where $Y$ is taut. Then $f$ is partially proper if and only if the maps $f_n = f\times {\rm id}_{\mb{A}^n_k} \colon Y\times \mb{A}^n_k\to  X\times \mb{A}^n_k$ are specializing for all $n\geqslant 0$.
\end{lem}

\begin{proof}
The only if direction is clear, since a partially proper map is specializing and being partially proper is stable under base change. To prove the converse, we first claim that $f$ is partially proper if and only if for all irreducible components $Z$ of $Y$ the map $Z\to X$ is proper. Since $Z$ is quasi-compact, evidently if $Y\to X$ is partially proper then $Z\to X$ is partially proper and thus proper (see \stacks{0BX5}). The converse is clear by thinking about the valuative criterion of properness for $Y\to X$ since for any valuation ring $V$ with fraction field $K$, a map $\Spec(K)\to Y$ must factorize through a unique irreducible component of $Y$.

From the above we see that it suffices to show that $Z\to X$ is proper for all irreducible components $Z$ of $Y$. But, since $Z\to X$ is of finite type, by \stacks{0205} it suffices to show that $Z\times\mb{A}^n_k \to X\times \mb{A}^n_k$ is specializing for all $n\geqslant 0$. Since $Z$ is a closed subscheme of $Y$, a moment's thought reveals that this is implied by the fact that the maps $f_n$ are specializing for all $n\geqslant 0$.
\end{proof}

\begin{proof}[Proof of Lemma~\ref{lem:pp-generic-special}]
Since $f_\eta$ is separated, so is $f_k$, by \cite[Proposition~4.7]{BoschLutkebohmertI}. Since $\mf{X}_\eta$ is quasi-paracompact and quasi-separated, it is taut, and hence so is $\mf{Y}_\eta$ and therefore also $\mf{Y}_k$.  By Lemma~\ref{lem:universally-spec}, it suffices to show that $\mf{Y}_k\times\mb{A}^n_k\to \mf{X}_k\times\mb{A}^n_k$ is specializing for all $n\geqslant 0$. Replacing $f\colon \mf{Y}\to\mf{X}$ with $\mf{Y}\times \widehat{\mb{A}}^n_{\cO_K}\to\mf{X}\times \widehat{\mb{A}}^n_{\cO_K}$, we are reduced to showing that $f_k$ itself is specializing.  We have a commutative square of topological spaces
\[ 
    \xymatrix{
        \mf{Y}_\eta \ar[r]^{\op{sp}_{\mf{Y}}} \ar[d]_{f_\eta} & \mf{Y}_k \ar[d]^{f_k} \\
        \mf{X}_\eta \ar[r]_{\op{sp}_{\mf{X}}} & \mf{X}_k  \\
    }
\]
Here, the horizontal arrows are specializing (because they are closed) and surjective. If the left arrow is specializing, then so is the diagonal composition, and then it is straightforward to check that the right arrow must be specializing as well. 
\end{proof}

\subsection{Tube open subsets}
\label{ss:tube}

Let $\mf{X}$ be a formal scheme of finite type over $\cO_K$ and let $Z\subseteq |\mf{X}|$ be a closed subset. We then define the \emph{tube open subset} of $\mf{X}_\eta$ associated to $Z$, denoted $T(\mf{X}|Z)$, to be the open subset $\op{sp}_\mf{X}^{-1}(Z)^\circ$, i.e.\ the topological interior of $\op{sp}_\mf{X}^{-1}(Z)$. With the notation $]Z[$, such opens were first considered by Berthelot in the context of rigid cohomology \cite{Berthelot}. The following topological properties of tube open subsets of $\mf{X}_\eta$ are important for the proof of Theorem \ref{thm:specialization}.

\begin{prop} \label{prop:tube-properties} 
    Let notation be as above.
    \begin{enumerate}
        \item The tube open subset $T(\mf{X}|Z)$ is an overconvergent open subset.
        \item The tube open subset $T(\mf{X}|Z)$ contains every maximal point of $\op{sp}_\mf{X}^{-1}(Z)$.
        \item For any cover $\{Z_i\}$ of $|\mf{X}|$ by closed subsets, the set of corresponding tube open subsets $\{T(\mf{X}|Z_i)\}$ forms an overconvergent open cover of $\mf{X}_\eta$.
    \end{enumerate}
\end{prop}

\begin{proof}
The proof of the first statement follows by combining \cite[Chapter 0, Proposition 2.3.15]{FujiwaraKato} and \cite[Chapter II, Proposition 4.2.5]{FujiwaraKato}. The second statement follows by combining \cite[Chapter II, Proposition 4.2.9]{FujiwaraKato} and \cite[Chapter II, Proposition 4.2.10 (2)]{FujiwaraKato}. The final statement is clear since $\bigcup_i T(\mf{X}|Z_i)$ is an overconverent open subset of $\mf{X}_\eta$ by the first statement, but also clearly contains every maximal point by the second statement, and thus must be all of $\mf{X}_\eta$ as desired.
\end{proof}

We now relate the tube open subset $T(\mf{X}|Z)$ with the generic fiber $(\widehat{\mf{X}}_{Z})_\eta$ of the formal completion of $\mf{X}$ along $Z$.

\begin{prop} \label{prop:tube-completion-comparison}
    Let $\mf{X}$ be a formal scheme of finite type over $\cO_K$ and let $Z\subseteq |\mf{X}|$ be a closed subset. The map of rigid $K$-spaces  $(\widehat{\mf{X}}_Z)_\eta\to \mf{X}_\eta$ induced by $\widehat{\mf{X}}_Z\to \mf{X}$ is an open immersion with image $T(\mf{X}|Z)$.
\end{prop}

\begin{proof} 
Let us first verify that $(\widehat{\mf{X}}_Z)_\eta\to \mf{X}_\eta$ is an open immersion. By working on an affine open cover we may assume that $\mf{X}=\Spf(A)$ and that $J$ is an admissible ideal of $A$ such that $\underline{\Spf(A/J)}=Z$. But, we know that $\widehat{\mf{X}}_Z=\Spf(B)$ where $B$ is the $J$-adic completion of $A$. One can then directly check that, in the notation of Construction \ref{const:generic-fiber-affine} that each $(B(J^n),B(J^n)^+)$ is a rational localization of $(A_K, A_K^\circ)$ and thus we see from Lemma \ref{lem:concrete-generic-fiber-description} that the map $\mf{X}_\eta\to \Spa(A_K)=\Spf(A)_\eta$ is an open immersion.

So, to show our desired claim it suffices to show that for any quasi-compact and quasi-separated open subset $U$ of $\mf{X}_\eta$ that the map $j\colon U\to\mf{X}_\eta$ factorizes through $T(\mf{X}|Z)$ if and only if it factorizes through $(\widehat{\mf{X}}_Z)_\eta$. But $j$ factorizes through $T(\mf{X}|Z)$ if and only if there exists a morphism of admissible formal schemes $\mf{j}\colon \mf{U}\to \mf{X}$ such that $j=\mf{j}_\eta$ such that $\mf{j}(|\mf{U}|)\subseteq Z$. But, by combining Proposition \ref{prop:completion-functor-description-2} and Proposition \ref{lem:generic-fiber-functor-of-points-1} this is also precisely the condition for $j$ to factorize through $(\widehat{\mf{X}}_Z)_\eta$.
\end{proof}

\section{The specialization map}
\label{s:spec}

In this section we produce the specialization map from the de Jong fundamental group of the rigid generic fiber of a formal scheme to the pro-\'etale fundamental group of its special fiber.

\subsection{Geometric coverings} 

We briefly recall the theory of the pro-\'etale fundamental group \cite[\S7]{BhattScholze}, which is based on the following notion of a covering space. 

\begin{definition}[{\cite[7.3.1(3)]{BhattScholze}}] \label{lem:geom-cover-geom-unib}
    Let $X$ be a locally topologically Noetherian scheme. A morphism of schemes $Y\to X$ is a \emph{geometric covering} if it is \'etale and satisfies the valuative criterion of properness (see \stacks{03IX}). We denote by $\Cov_X$ the category of geometric coverings of $X$, treated as a full subcategory of $\Et_X$.
\end{definition}

The following lemma plays a pivotal role in the construction of the specialization map. 

\begin{lem}[{See \cite[Lemma~7.4.10]{BhattScholze} and its proof}] \label{lem:normal-BS}
    Let $X$ be a locally topologically Noetherian scheme. If $X$ is geometrically unibranch (\stacks{0BQ2}, e.g.\, normal), then every geometric covering of $X$ is the disjoint union of finite \'etale coverings of $X$.
\end{lem}

By \cite[Lemma 7.4.1]{BhattScholze}, for a connected and locally topologically Noetherian scheme $X$ with a geometric point $\ov{x}$, the pair $(\Cov_X,F_{\ov{x}})$ where $F_{\ov{x}}\colon \Cov_X\to\Set$ is the fiber functor at $\ov x$, is a tame infinite Galois category (see \cite[\S7.2]{BhattScholze}). Following \cite[Definition 7.4.2]{BhattScholze} we denote the fundamental group of $(\Cov_X,F_{\ov{x}})$ by $\pi_1^\proet(X,\ov{x})$ and call it the \emph{pro-\'etale fundamental group} of $X$. By definition, it is a Noohi topological group (see \cite[\S7.1]{BhattScholze}) for which the induced functor
\[ 
    F_{\ov{x}} \colon \Cov_X \isomto \pi_1^\proet(X, \ov x)\text{-}\cat{Sets}
\]
is an equivalence, where the target is the category of (discrete) sets with a continuous action of $\pi_1^\proet(X, \ov x)$.

\subsection{Coverings of rigid-analytic spaces}\label{ss:rigid-covering-spaces}

Next, we recall various notions of `covering spaces' of a rigid space, as defined in \cite{ALY1}. Here, we will only use de Jong covering spaces, while the more general $\adm$-covering spaces and $\et$-covering spaces will be needed in the next section. 

Let $X$ be a rigid $K$-space, and let $\tau\in\{\oc,\adm,\et\}$. By a \emph{$\tau$-cover} of $X$ we shall mean a surjective morphism $U\to X$ such that
\begin{itemize}
    \item if $\tau=\oc$, then $U\to X$ is the disjoint union of overconvergent open immersions into $X$, 
    \item if $\tau =\adm$, then $U\to X$ is the disjoint union of open immersions into $X$,
    \item if $\tau=\et$, then $U\to X$ is \'etale.
\end{itemize}

\begin{definition} \label{def:tau-cov}
    Let $X$ be a rigid $K$-space, and let $\tau\in\{\oc,\adm,\et\}$. By a \emph{$\tau$-covering space} of $X$ we mean a morphism $Y\to X$ for which there exists a $\tau$-cover $U\to X$ such that $Y_U\to U$ is the disjoint union of finite \'etale coverings of $U$. We denote by $\Cov_X^\tau$ the full subcategory of $\Et_X$ consisting of $\tau$-covering spaces, and by $\UCov_X^\tau$ the full subcategory of $\Et_X$ consisting of arbitrary disjoint unions of objects of $\Cov_X^\tau$.
\end{definition}

When $X$ is a taut rigid $K$-space (see \cite[Definition 5.1.2]{Huberbook}), then $\Cov_X^\oc$ is naturally equivalent (via the functor in \cite[\S8.3]{Huberbook}) to the category of \'etale covering spaces of the associated Berkovich space $X^\mathrm{Berk}$ considered in \cite{deJongFundamental}. Thus, we shall refer to $\text{oc}$-covering spaces as \emph{de Jong covering spaces}.

For a geometric point $\ov{x}$ let $F_{\ov{x}}\colon \Et_X\to\cat{Set}$ be the fiber functor $F_{\ov{x}}(Y)=\pi_0(Y_{\ov{x}})$. For a connected rigid $K$-space $X$ and $\tau=\oc$, de Jong essentially showed that the pair $(\UCov^\oc_X,F_{\ov{x}})$ is a tame infinite Galois category (see \cite[Theorem~2.9]{deJongFundamental}). More generally, for a connected rigid $K$-space $X$ one knows from \cite[Theorem 5.2.1]{ALY1} that the pair $(\UCov^\tau_X,F_{\ov{x}})$ is a tame infinite Galois category 
for each $\tau\in\{\oc,\adm,\et\}$. In particular, we have an equivalence of categories 
\[
    F_{\ov x}\colon \UCov_X^\tau \isomto \pi_1(\UCov^\tau_X, F_{\ov{x}})\text{-}\cat{Sets}.
\]
We denote the fundamental group of $(\UCov^\oc_X,F_{\ov{x}})$ by $\pi_1^\dJ(X,\ov{x})$ and call it the \emph{de Jong fundamental group}. 

One of the upshots of the material discussed and developed in \S\ref{s:prelims} is the ability to prove the following criterion for when the rigid generic fiber of an \'etale map of formal schemes is a de Jong covering space.

\begin{prop}[Overconvergence criterion] \label{prop:spec-good-case}
    Let $\mf{Y}\to\mf{X}$ be a morphism of admissible formal schemes over $\h_K$ with $\mf{X}$ quasi-compact and such that $\mf{Y}_\eta\to\mf{X}_\eta$ is \'etale. Suppose that for each irreducible component $Z$ of $\mf{X}_k$ one has that $\mf{Y}_Z$ is a disjoint union of finite coverings of $Z$. Then, $\mf{Y}_\eta\to\mf{X}_\eta$ is a de Jong covering space.
\end{prop}

\begin{proof} 
Let $Z_1, \ldots, Z_r$ be the irreducible components of $\mf{X}_k$ and let $\widehat{\mf{X}}_i$ for $i=1, \ldots, r$ be the formal completions of $\mf{X}$ along $Z_i$. Since the pullback of $\mf{Y}_{\widehat{\mf{X}}_i}\to \widehat{\mf{X}}_i$ to $Z_i$ is the disjoint union of finite maps, we know by Proposition~\ref{prop:alt-refined-topological-invariance} that the same holds for $\mf{Y}_{\widehat{\mf{X}}_i}\to \widehat{\mf{X}}_i$. We deduce from Proposition~\ref{prop:generic-fiber-props} and Proposition~\ref{prop:tube-completion-comparison} that the pullback of $\mf{Y}_\eta$ to $T(\mf{X}|Z_i)$ is the disjoint union of finite \'etale coverings (\'etale by our assumption on $\mf{Y}_\eta\to\mf{X}_\eta$) for all $i$. Since $\{T(\mf{X}|Z_i)\}$ is an overconvergent open cover of $\mf{X}_\eta$ by Proposition \ref{prop:tube-properties} we are done.
\end{proof}

\subsection{The \'etale site of a formal scheme}

We now recall the topological invariance of the etale site of a formal scheme.

\begin{prop} \label{topological invariance formal schemes}     
    Let $\mf{X}'\to \mf{X}$ be a universal homeomorphism of formal schemes. Then, the induced functor $\Et_\mf{X}\to \Et_{\mf{X}'}$ is an equivalence. In particular, the functor $\Et_{\mf{X}}\to\Et_{\underline{\mf{X}}}$ is an equivalence, and if $\mf{X}$ is an adic formal scheme over $\cO_K$, then $\Et_{\mf{X}}\to\Et_{\mf{X}_k}$ is an equivalence.
\end{prop}

\begin{proof} 
Note that since $\mf{X}'\to \mf{X}$ is a universal homeomorphism, so then is the morphism of schemes $\underline{\mf{X}}'\to\underline{\mf{X}}$. So, the induced morphism $\Et_{\underline{\mf{X}}}\to\Et_{\underline{\mf{X}}'}$ is an equivalence by \stacks{04DZ}. This clearly then reduces us to showing that, in general, the map $\Et_{\mf{X}}\to\Et_{\underline{\mf{X}}}$ is an equivalence. This assertion is clearly local on $\mf{X}$ and so we may assume without loss of generality that $\mf{X}$ is quasi-compact and quasi-separated and hence has an ideal sheaf of definition $\mc{I}$. Since the map of schemes $\underline{\mf{X}}\to \mf{X}(\mc{I})$ induces an equivalence on \'etale sites by loc.\@ cit.\@ we are reduced to showing that the map $\Et_\mf{X}\to\Et_{\mf{X}(\mc{I})}$ is an equivalence. But, this follows by combining loc.\@ cit.\@ and \cite[Chapter I, Proposition 1.4.2]{FujiwaraKato} since $\mf{X}=\varinjlim \mf{X}(\mc{I}^n)$ and each $\mf{X}(\mc{I}^n)\to \mf{X}(\mc{I}^{n+1})$ is a thickening of schemes.
\end{proof}

\subsection{Statement of the main result}

Let $\mf{X}$ be a formal scheme locally of finite type over $\cO_K$ and let $X=\mf{X}_\eta$ be its rigid generic fiber. We have the functors between the categories of \'etale objects
\[ 
	\xymatrix{
		\Et_{X} & \Et_{\mf{X}} \ar[l]_{(-)_\eta} \ar[r]^{(-)_k} & \Et_{\mf{X}_k}
	} 
\]
where $(-)_\eta$ (resp.\ $(-)_k$) is the functor which maps $\mf{Y}\to \mf{X}$ to $\mf{Y}_{\eta}\to \mf{X}_{\eta}=X$ (resp.\@ $\mf{Y}_k\to\mf{X}_k$). The functor $(-)_k$ is an equivalence by Proposition~\ref{topological invariance formal schemes}. Let us denote by
\[ 
	u\colon \Et_{\mf{X}_k} \to \Et_{X}
\]
the composition $u=(-)_\eta\circ (-)_k^{-1}$. The main result of this section can then be stated as follows.

\begin{thm} \label{thm:specialization}
	Let $\mf{X}$ be a quasi-paracompact formal scheme locally of finite type over $\cO_K$ and let $X=\mf{X}_\eta$. Then, the functor $u$ maps $\Cov_{\mf{X}_k}$ into $\Cov^\oc_X$.
\end{thm}

Let us interpret this result in terms of fundamental groups. To this end, we first need to discuss compatible choices of base points for $\mf{X}_k$ and $X$. Let $L$ be an algebraically closed non-archimedean field and let $\ov{\xi}\colon \Spf(\cO_L)\to \mf{X}$ be a morphism of formal schemes, inducing geometric points $\ov{\xi}_\eta\colon \Spa(L)\to X$ and $\ov{\xi}_0\colon \Spec(\ell)\to \mf{X}_k$, where $\ell = \cO_L/\mf{m}_L$ is the residue field of $L$. In this situation, if $\mf{Y}\to \mf{X}$ is an \'etale morphism, then the induced maps
\[ 
    \Hom_X(\Spa(L), \mf{Y}_\eta) \leftarrow \Hom_{\mf{X}}(\Spf(\cO_L), \mf{Y}) \to \Hom_{\mf{X}_k}(\Spec(\ell), \mf{Y}_k)
\]
are bijective. To see this, we may replace $\mf{X}$ with $\Spf(\cO_L)$, and then it suffices to show that every \'etale morphism $\mf{Y}\to\mf{X}$ is a disjoint union of copies of $\mf{X}$. This follows from Proposition~\ref{topological invariance formal schemes} since  $\underline{\mf{X}}=\Spec(\ell)$.

\begin{cor}\label{cor:specialization}
	Let $\mf{X}$ be a quasi-paracompact formal scheme locally of finite type over $\cO_K$ whose special fiber $\mf{X}_k$ and rigid generic fiber $X=\mf{X}_\eta$ are both connected. Then for a compatible choice of base points as above, the functor $u$ induces a continuous homomorphism of Noohi groups
	\[ 
		\pi_1^\dJ(X, \ov{\xi}_\eta) \to \pi_1^\proet(\mf{X}_k, \overline{\xi}_0).
	\]
	Its image is dense if $\mf{X}$ is $\eta$-normal (see Appendix~\ref{app:eta-normal}).
\end{cor}

\begin{proof} 
Given Theorem \ref{thm:specialization}, the existence of the specialization map is immediate from \cite[Theorem 7.2.5 2.]{BhattScholze}, since by the discussion above there is a natural identification of geometric fibers $F_{\ov{\xi}_\eta}(u(Y)) \simeq F_{\ov{\xi}_0}(Y)$.

The only part left to be justified is the claim of dense image. Note that by \cite[Proposition 2.37.(2)]{LaraI} and the proof of \cite[Theorem 5.2.1]{ALY1} it suffices to show that if $Y\to \mf{X}_k$ is a connected geometric covering and $\mf{Y}\to\mf{X}$ its unique \'etale deformation, then $\mf{Y}_\eta=u(Y)$ is connected. But, since $\mf{Y}\to \mf{X}$ is \'etale, the formal scheme $\mf{Y}$ is $\eta$-normal by Proposition~\ref{prop:etale-eta-normal}, and we use Lemma~\ref{lem:connected generic fiber} to conclude.
\end{proof}

The natural converse to Theorem \ref{thm:specialization} is relatively straightforward given our setup, and we obtain the following (cf.\@ \cite[Chapter II, Theorem 7.5.17]{FujiwaraKato}).

\begin{cor} \label{cor:tfae}
    Let $\mf{Y}\to\mf{X}$ be an \'etale map of admissible formal schemes over $\h_K$, with $\mf{X}$ quasi-paracompact. Then, the following are equivalent:
    \begin{enumerate}[(a)]
    \item $\mf{Y}_k\to\mf{X}_k$ is a geometric covering,
    \item $\mf{Y}_\eta\to\mf{X}_\eta$ is a de Jong covering space,
    \item $\mf{Y}_\eta\to\mf{X}_\eta$ is partially proper.
    \end{enumerate}
\end{cor}

\begin{proof} 
That (a) implies (b) is precisely the content of Theorem \ref{thm:specialization}. To see that (b) implies (c) is simple since being partially proper is local on the target, and evidently a disjoint union of finite \'etale coverings is partially proper. That (c) implies (a) is then given by Lemma \ref{lem:pp-generic-special}. 
\end{proof}

\begin{rem} \label{rem:etale-nonsplit}
The following example shows that geometric coverings might not be the disjoint union of finite \'etale coverings \'etale locally on the target, illustrating the essential difficulty behind Theorem~\ref{thm:specialization}. Let $k$ be an algebraically closed field, and let $X$ be the nodal curve obtained by gluing two copies $X^+$, $X^-$ of $\mb{G}_{m,k}$ at a closed point $x$. For $n\geqslant 0$, let $Y^\pm_n\to X^\pm$ be the connected cyclic covering of degree equal to the $n$-th prime number. Let $Y^+ = \coprod_{n\geqslant 0} Y_n^+$ and $Y^-=\coprod_{n>0} Y_n^-$, and let $Y\to X$ be the geometric covering of $X$ with $Y|_{X^\pm}\simeq Y^\pm$ obtained by identifying the fibers of $Y^\pm$ at $x$ as in the picture below.
\begin{displaymath}
    \mathrlap{\underbrace{\phantom{\bullet \quad \bullet \quad}}_{Y_0^+}}
    \mathrlap{\overbrace{\phantom{\ldots \bullet \quad \bullet \quad }}^{Y_1^-}}
    \bullet \quad \bullet \quad
    \mathrlap{\underbrace{\phantom{\bullet \quad \bullet \quad \bullet \quad}}_{Y_1^+}}
    \bullet \quad 
    \mathrlap{\overbrace{\phantom{\bullet \quad \bullet \quad \bullet \quad \bullet \quad \bullet \quad}}^{Y_2^{-}}}
    \bullet \quad \bullet \quad 
    \mathrlap{\underbrace{\phantom{\bullet \quad \bullet \quad \bullet \quad \bullet \quad \bullet \quad}}_{Y_2^+}}
    \bullet \quad \bullet \quad \bullet \quad
    \mathrlap{\overbrace{\phantom{\bullet \quad \bullet \quad \bullet \quad \bullet \quad \bullet \quad \bullet \quad \bullet \quad }}^{Y_3^{-}}}
    \bullet \quad \bullet \quad 
    \mathrlap{\underbrace{\phantom{\bullet \quad \bullet \quad \bullet \quad \bullet \quad \bullet \quad \bullet \quad  \bullet \quad }}_{Y_3^+}}
    \bullet \quad \bullet 
    \quad \bullet \quad \bullet \quad \bullet \quad 
    \mathrlap{\overbrace{\phantom{\bullet \quad \bullet \quad \bullet \ldots \quad}}^{Y_4^{-}}}
    \bullet \quad \bullet \quad 
    \mathrlap{\underbrace{\phantom{\bullet \ldots \quad}}_{Y_4^+}}
    \bullet  \ldots
\end{displaymath}  
We sketch the proof that $Y\to X$ is not the disjoint union of finite \'etale coverings in any \'etale neighborhood of $x$. It suffices to observe that for a connected quasi-compact \'etale neighborhood $U^\pm\to X^\pm$ of $x$, the pull-backs $Y^\pm_n\times_{X^\pm} U^\pm\to U^\pm$ are connected for almost all $n$. But $U^\pm$ is an integral curve, and $Y^\pm_n\times_{X^\pm} U^\pm$ will be connected as long as the $n$-th prime does not divide the degree of the field extension $k(U^\pm)/k(X^\pm)$. 
\end{rem}

\subsection{The proof}

The strategy of proof of Theorem~\ref{thm:specialization} is to utilize the overconvergence criterion (Proposition~\ref{prop:spec-good-case}). In particular, if all of the irreducible components of $\mf{X}_k$ are normal then we are done by Lemma~\ref{lem:normal-BS}. Of course, this condition on the components of $\mf{X}_k$ will rarely be satisfied for a given formal scheme $\mf{X}$. 

To remedy this we perform a suitable admissible blowup $\mf{X}'\to \mf{X}$ (Proposition~\ref{prop:nice-blowup}). Namely, even though the normalization $n(\mf{X}_k)\to \mf{X}_k$ usually does not extend to a proper birational map to $\mf{X}$, in most situations it is a blowup, and in any case it is dominated by some blowup $W\to \mf{X}_k$. Then, blowing up the same center on $\mathfrak{X}$ gives an admissible blowup $\mathfrak{X}'\to\mathfrak{X}$ such that the strict transform of $\mf{X}_k$ factors through $n(\mf{X}_k)$. The restriction of a given geometric covering of $\mf{X}_k$ to $W\subseteq \mf{X}'_k$ is thus the disjoint union of finite \'etale maps. However, the blowup $\mathfrak{X}'\to \mathfrak{X}$ introduces new irreducible components of $\mathfrak{X}'_k$ whose images in $\mf{X}_k$ have positive codimension. We handle those in a similar way, and the proof proceeds by induction.

\begin{lem} \label{lemma:blowup}
	Let $\mf{X}$ be a quasi-compact admissible formal scheme over $\cO_K$ and let $Z_1, \ldots, Z_r$ be closed integral subschemes of $\mf{X}_k$ such that $Z_i\not\subseteq Z_j$ for $i\neq j$. Then there exists an admissible blowup $\mf{X}'\to\mf{X}$ whose center does not contain any $Z_i$ and such that the strict transforms $Z'_i\to Z_i$ in $\mf{X}'$ are birational and factor through the normalization of $Z_i$ for each $i=1, \ldots, r$.
\end{lem}

Birationality of $Z'_i\to Z_i$ means that the blowup $\mf{X}'\to\mf{X}$ is an isomorphism in a neighborhood of the generic point of $Z_i$. 

\begin{proof}[Proof of Lemma \ref{lemma:blowup}]
For every $i$, there is a non-zero quasi-coherent ideal sheaf $J_i\subseteq \cO_{Z_i}$ whose blowup $Z''_i\to Z_i$ factors through the normalization of $Z_i$. Indeed, if $U\subseteq Z_i$ is the normal locus \cite[7.8.3(iv)]{EGAIV_2} then the normalization is a proper map which is an isomorphism over $U$, and hence is dominated by a $U$-admissible blowup (see \stacks{080K}) by \cite[I 5.7.12]{RaynaudGruson} (see also \stacks{081T}). We also denote by $J_i\subseteq \h_{\mf{X}_k}$ the induced quasi-coherent ideal sheaf. By Lemma \ref{lem:lifting-ideal-sheaf} below there exists admissible ideal sheaves $\mc{J}_i\subseteq \cO_\mf{X}$ with $\mc{J}_i\cdot \cO_{\mf{X}_k} = J_i$. If $\mf{X}'_i\to\mf{X}$ is the admissible blowup of $\mc{J}_i$, then $Z''_i$ is the strict transform of $Z_i$ in $\mf{X}'_i$.

Let $\mc{J} = \prod_{i=1}^r \mc{J}_i$ which is again an admissible ideal sheaf by \cite[Chapter I, Proposition 3.7.9]{FujiwaraKato}, and let $\mf{X}'\to \mf{X}$ be the corresponding admissible blowup. Since the generic point of $Z_i$ does not belong to any other $Z_j$, the map $\mf{X}'\to\mf{X}$ is an isomorphism generically on $Z_i$, and in particular the blowup center $V(\mc{J})$ does not contain any $Z_i$. If $Z'_i$ is the strict transform of $Z_i$ in $\mf{X}'$, then $Z'_i\to Z_i$ factors through $Z''_i\to Z_i$ and hence through the normalization of $Z_i$.
\end{proof}

\begin{lem} \label{lem:lifting-ideal-sheaf}
    Let $\mf{X}$ be a formal scheme of finite type over $\cO_K$ and let $\mc{J}\subseteq \cO_{\mf{X}_k}$ be a coherent ideal sheaf. Then there exists an admissible ideal sheaf $\mc{I}\subseteq \cO_{\mf{X}}$ such that $\mc{I}\cdot \cO_{\mf{X}_k} = \mc{J}$.
\end{lem}

\begin{proof}
We start the proof with an observation. Suppose that $\mf{X}=\Spf(A)$ and let $J=(a_1, \ldots, a_r)\subseteq A_k$ be the ideal corresponding to $\mc{J}$. If $a'_i\in A$ are liftings of the $a_i$ and $I' = (a'_1, \ldots, a'_r, \varpi) \subseteq A$, then $I'$ is an admissible ideal and $I'A_k=J$, so in particular the result holds if $\mf{X}$ is affine. Moreover, if $I'=(a'_1, \ldots, a'_r, \varpi)$ and $I''=(b'_1, \ldots, b'_s,\varpi)$ are two admissible ideals of $A$ with $I' A_k = J = I''A_k$, then we can write $b'_i = \sum_j \alpha_{ij} a'_j + \varpi_i \gamma_i$ with $\alpha_{ij}, \gamma_i\in A$ and $\varpi_i\in \mf{m}_K$. It follows that if $(\varpi,\varpi_1, \ldots, \varpi_s)\cO_K = (\varpi')$ then $I''\subseteq I'+(\varpi')$. Swapping the roles of $I'$ and $I''$, we get $I'\subseteq I''+(\varpi'')$ for some pseudouniformizer $\varpi''$. Taking $(\omega)=(\varpi', \varpi'')$, we see that $I'+(\omega) = I''+(\omega)$ for a suitably chosen pseudouniformizer $\omega$. 

For the general case, cover $\mf{X}$ by finitely many affine opens $\mf{U}_i=\Spf(A_i)$ and cover the pairwise intersections $\mf{U}_i\cap \mf{U}_{i'}$ by finitely many affine opens $\mf{U}_{ii'j}=\Spf(A_{ii'j})$. By the previous paragraph, we can find admissible ideals $I'_i\subseteq A_i$ such that $I'_i(A_i)_k$ is the ideal corresponding to $\mc{J}|_{(\mf{U}_i)_k}$. Moreover, for each triple $(i,i',j)$ there exists a pseudouniformizer $\omega_{ii'j}$ such that $I'_i+(\omega_{ii'j})$ and $I'_{i'}+(\omega_{ii'j})$ generate the same ideal in $A_{ii'j}$. If $(\omega)=(\omega_{ii'j})_{(i,i',j)}$, then the admissible ideals $I_i = I'_i + (\omega) \subseteq A_i$ glue to give the desired admissible ideal $\mc{I}$ sheaf.  
\end{proof}

Next, we show that we may iteratively apply Lemma~\ref{lemma:blowup} to produce an admissible blowup of $\mf{X}'$ of $\mf{X}$ whose irreducible components factor through the normalizations of their images in $\mf{X}$.

\begin{prop} \label{prop:nice-blowup}
	Let $\mf{X}$ be a quasi-compact admissible formal scheme over $\cO_K$. Then there exists an admissible blowup $\mf{X}'\to \mf{X}$ such that for every irreducible component $Z'$ of $\mf{X}'_k$ with image $Z\subseteq \mf{X}_k$, the map $Z'\to Z$ factors through the normalization of $Z$. 
\end{prop}

\begin{proof}
We prove by induction on $n\geqslant 0$ the following statement: there exists an admissible blowup $\mf{X}'_n\to \mf{X}$ such that the assertion of the proposition holds for all $Z'$ such that $\op{codim} Z\leqslant n$. Here codimension means the Krull dimension of the local ring $\cO_{\mf{X}_k, \eta_Z}$ where $\eta_Z$ is the generic point of $Z$. This claim for $n=\dim \mf{X}_k$ will then imply the required assertion. 

Starting with $n=0$, let $Z_1, \ldots, Z_r$ be the irreducible components of $\mf{X}_k$. Clearly $Z_i$ is not contained in $Z_j$ for $i\neq j$. Let $\mf{X}'_0\to \mf{X}$ be the admissible blowup given by Lemma~\ref{lemma:blowup}. This blowup is birational at the generic points of the $Z_i$, and hence every irreducible component which is not the strict transform of some $Z_i$ has image of codimension at least one. The claim is thus proved for $n=0$.

For the induction step from $n$ to $n+1$, suppose that we have constructed $\mf{X}'_n\to \mf{X}$ as above. Let $Z'_1, \ldots, Z'_r$ be the irreducible components of $(\mf{X}'_{n})_k$ whose images $Z_1, \ldots, Z_r$ in $\mf{X}_k$ have codimension exactly $n+1$. Again, $Z_i$ does not properly contain another $Z_j$ (we could have $Z_i=Z_j$ for some $i\neq j$). Apply Lemma~\ref{lemma:blowup} to obtain an admissible blowup $\mf{X}''_n\to \mf{X}$ relative to an admissible ideal sheaf $\mc{J}$ with $\op{codim} V(J)>n+1$ where $J=\mc{J}\cdot\cO_{\mf{X}_k}$. If $\mf{X}'_n$ is the blowup relative to an admissible ideal sheaf $\mc{I}$ and $\mf{X}''_n$, let us take $\mf{X}'_{n+1}\to \mf{X}$ to be the blowup relative to $\mc{I}\cdot \mc{J}$, which is also the blowup of $\mf{X}'_n$ relative to $\mc{J}\cdot\cO_{\mf{X}'}$. 

We check that $\mf{X}'_{n+1}\to\mf{X}$ satisfies the assertion for codimension $n+1$. Let $Z''$ be an irreducible component of $(\mf{X}'_{n+1})_k$, and let $Z'$ and $Z$ be its images in $(\mf{X}'_n)_k$ and $\mf{X}_k$, respectively. There are four cases to consider:
\begin{enumerate}[(1)]
    \item $Z'$ is an irreducible component of $(\mf{X}'_n)_k$ and
        \begin{enumerate}[(a)]
            \item $\op{codim} Z\leqslant n$, 
            \item $\op{codim} Z = n+1$, 
            \item $\op{codim} Z > n+1$, or
        \end{enumerate}
    \item $Z'$ is not an irreducible component of $(\mf{X}'_n)_k$.
\end{enumerate}
In case (1a) we are done by the assumed property of $\mf{X}'_n$, and in case (1c) there is nothing to show. In case (1b),  we have $Z'=Z'_i$ for some $i$. Since the generic point of $Z=Z_i$ is not in $V(J)$, the generic point of $Z'$ is not in $V(J\cdot \cO_{(\mf{X}'_n)_k})$. Since $\mf{X}'_{n+1}\to \mf{X}'_n$ is the blowup at $\mc{J}\cdot \cO_{\mf{X}'_n}$, it is an isomorphism in a neighborhood of the generic point of $Z'$, and we conclude that $Z''$ is the strict transform of $Z'$ in $\mf{X}'_{n+1}$. That is, $Z'' = \op{Bl}_{\mc{J}\cdot \cO_{Z'}}(Z')$, and so by functoriality it maps to the strict transform $\op{Bl}_{\mc{J}\cdot \cO_Z}(Z)$ of $Z$ in $\mf{X}''_n$. By construction, the latter factors through the normalization of $Z$, and we are done. See the diagram below.
\[
    \xymatrix{
        \mf{X}'_{n+1} \ar[rr]^{\mc{J}\cdot\cO_{\mf{X}'_n}} \ar[d]_{\mc{I}\cdot\cO_{\mf{X}''_n}} & & \mf{X}'_n \ar[d]^{\mc{I}} \\
        \mf{X}''_n \ar[rr]_{\mc{J}} & & \mf{X}  
    }
    \qquad
    \xymatrix{
        Z'' = \op{Bl}_{\mc{J}\cdot \cO_{Z'}}(Z') \ar[rr] \ar[d] & &  Z' = Z'_i \ar[d] \\
        \op{Bl}_{\mc{J}\cdot \cO_Z}(Z) \ar[r] & n(Z) \ar[r] & Z=Z_i 
    }
\]

Finally, in case (2), $Z''$ is contained in the exceptional locus of $\mf{X}'_{n+1}\to \mf{X}'_n$, which is contained in the preimage of $V(\mc{J})$ in $\mf{X}'_n$. Thus $Z\subseteq V(J)$, and hence $\op{codim} Z\geqslant \op{codim} V(J)>n+1$.
\end{proof}

We are now ready to put the pieces together and present the proof of Theorem \ref{thm:specialization}.

\begin{proof}[Proof of Theorem~\ref{thm:specialization}]
We first reduce to the case $\mf{X}$ is of finite type. Since $\mf{X}$ is quasi-paracompact, its rigid generic fiber $X=\mf{X}_\eta$ is quasi-paracompact and quasi-separated, and hence taut by \cite[Chapter 0, Proposition 2.5.15]{FujiwaraKato}.  Therefore, for each point $x\in X$ there exists a quasi-compact open $U_x\in X$ which is an overconvergent neighborhood of $x$, i.e.\@, it contains an overconvergent open containing $x$ (see \cite[Proposition 1.4.5]{ALY1}). It then suffices to show that for any geometric covering $Y\to \mf{X}_k$ that $u(Y)$ is a de Jong covering space when restricted to $U_x$. Since $U_x$ is quasi-compact, by (the proof of) \cite[Lemma~8.4/5]{BoschLectures} there exists an admissible blowup $\mf{X}'\to \mf{X}$ such that $U_x=\mf{U}'_\eta$ for an open formal subscheme $\mf{U}'\subseteq\mf{X}'$. If the result holds for $\mf{U}'$ (which is of finite type) and the base change $Y' = Y\times_{\mf{X}_k} \mf{U}'_k$ of the given geometric covering $Y\to \mf{X}_k$, we deduce that $\mf{Y}'_\eta\to \mf{U}'_\eta = U_x$ is a de Jong covering, where $\mf{Y}'\to \mf{U}'$ is the unique \'etale lifting of $Y'\to \mf{U}'_k$. But $\mf{Y}'_\eta = u(Y)\times_X U_x$, and we conclude.

Suppose now that $\mf{X}$ is of finite type. Blowing up an ideal sheaf of definition, we may assume that $\mf{X}$ is admissible. Proposition~\ref{prop:nice-blowup} yields an admissible blowup $f\colon \mf{X}'\to \mf{X}$ such that for every irreducible component $Z'$ of $\mf{X}'_k$ the induced map $Z'\to \mf{X}'_k\to \mf{X}_k$ factors through a normal scheme $Z$ of finite type over $k$. By Lemma~\ref{lem:geom-cover-geom-unib} we know that $Y_{Z'}$ is the disjoint union of finite \'etale coverings of $Z'$.  Since $f_\eta$ is an isomorphism, we may identify $X'=\mf{X}_{\eta}$ with $X$ and we have a commutative diagram
\[ 
	\xymatrix{
		\mathbf{Et}_{X'} & \mathbf{Et}_{\mf{X}'_k} \ar[l]_{u'}  \\
		\mathbf{Et}_{X}  \ar@{=}[u] & \mathbf{Et}_{\mf{X}_k}\ar[l]_u \ar[u]_{f^*}. 
	}
\]
By Proposition~\ref{prop:spec-good-case}, the functor $u'$ maps the image of $f^*\colon \Cov_{\mf{X}_k}\to \Cov_{\mf{X}'_k}$ into $\Cov^\oc_{X'}=\Cov^\oc_X$, and hence $u$ sends $\Cov_{\mf{X}_k}$ into $\Cov^\oc_X$. 
\end{proof} 

\begin{rem} 
The authors suspect that the above results hold in a greater generality. Namely, let $\mf{X}$ be a universally rigid-Noetherian formal scheme (in the sense of \cite[Chapter I, Definition 2.1.7]{FujiwaraKato}) such that $\underline{\mf{X}}$ is topologically Noetherian and universally Japanese (see \stacks{032R}). There is then still a functor $\Et_{\underline{\mf{X}}}$ to $\Et_{\mf{X}_\mathrm{rig}}$ (where $\mf{X}_\mathrm{rig}$ is as in \cite[Chapter II]{FujiwaraKato}), and the same methods should show that $\Cov_{\underline{\mf{X}}}$ is mapped into $\Cov^\oc_{\mf{X}_\mathrm{rig}}$. However, an interpretation of such a result in terms of fundamental groups is clearly missing, since we do not know if for $\mf{X}_{\rm rig}$ connected the pair $(\Cov^\oc_{\mf{X}_\mathrm{rig}}, F_{\ov x})$ is a tame infinite Galois category in this generality.
\end{rem}

\section{Tame coverings}
\label{s:tame}

In this section, we assume that the base non-archimedean field $K$ is discretely valued and we fix a uniformizer $\varpi$. We aim to show that, under certain smoothness hypotheses, the notion of tame (see Definition~\ref{tame def}) $\et$-covering spaces, $\adm$-covering spaces, and de Jong coverings of a rigid $K$-space all coincide (Theorem~\ref{thm:tame-comparison}). The method of proof mirrors that of Theorem~\ref{thm:specialization}. Namely we formulate a simple topological lemma (see Proposition \ref{prop:oc-criterion}) which will allow us to apply the overconvergence criterion of Proposition~\ref{prop:spec-good-case}. We then employ the link between tameness and logarithmic geometry (the logarithmic Abhyankar's lemma, Theorem~\ref{thm:log-abhy}) to verify that our topological lemma applies to suitable models of tame covering spaces.

\subsection{Preliminaries about finite maps and formal models}
\label{ss:tame-prelim}

The following lemma shows that, like in algebraic geometry (see \stacks{0BAK}), given a formal scheme $\mf{X}$ and a finite morphism of rigid $K$-spaces $Y\to\mf{X}_\eta$, one can `normalize $\mf{X}$ in $Y$,' producing a formal model $\mf{Y}\to\mf{X}$ which is again a finite morphism. There is a unique such map with the property that $\mf{Y}$ is $\eta$-normal (see Appendix~\ref{app:eta-normal} for the definition and basic properties of $\eta$-normality).

\begin{lem} \label{lem:normalization}
	Let $\mf{X}$ be an admissible formal scheme over $\cO_K$, let $X=\mf{X}_\eta$, and let $f\colon Y\to X$ be a finite morphism with $Y$ reduced. Then there exists a unique up to unique isomorphism finite morphism of formal schemes $\mf{f}\colon \mf{Y}\to\mf{X}$ with $f=\mf{f}_\eta$ and $\mf{Y}$ $\eta$-normal.
\end{lem}

\begin{proof}
Suppose that $\mf{f}\colon \mf{Y}\to \mf{X}$ is a finite map with $\mf{f}_\eta = f$ and $\mf{Y}$ $\eta$-normal. On the one hand, $\op{sp}_{\mf{Y},*}\cO_Y^+ = \cO_\mf{Y}$ because $\mf{Y}$ is $\eta$-normal. On the other, $\mf{Y}$ is the relative formal spectrum \cite[Chapter I, Proposition 4.2.6]{FujiwaraKato} of $\mf{f}_* \cO_\mf{Y} = \mf{f}_* (\op{sp}_{\mf{Y}, *} \cO_Y^+) = \op{sp}_{\mf{X}, *} (f_*\cO^+_Y)$. This implies the uniqueness, and that to show existence we need to check that $\op{sp}_{\mf{X}, *} (f_*\cO^+_Y)$ is an adically quasi-coherent and finitely generated $\cO_\mf{X}$-algebra. For an affinoid open $\Spf(A)$ in $X$, write $Y\times_X \Spa(A_K) = \Spa(B)$ for a finite and reduced $A_K$-algebra $B$. The value of the sheaf in question on $\Spf(A)$ is $B^\circ$.  Choose a surjection $\cO_K\langle x_1, \ldots, x_n\rangle\to A$. Since $K$ is discretely valued, by \cite[Corollary~6.4.1/6]{BGR} applied to $K\langle x_1, \ldots, x_n\rangle\to A_K$, the algebra $A_K^\circ$ is finite over $A$, and by the same result applied to $A_K\to B$, the algebra $B^\circ$ is finite over $A^\circ$ and hence also over $A$. Moreover the sheaf $\op{sp}_{\mf{X},*} (f_* \h_Y^+)$ is adically quasi-coherent since for $a$ in $A$ the above discussion, and the sheaf property for $\h_Y^+$, shows that $B^\circ\otimes_A A\langle a^{-1}\rangle$ is precisely $B^\circ\langle a^{-1}\rangle^\circ$.
\end{proof}

The following corollary extends the above result to the case of maps which are locally on the target the disjoint union of finite maps.

\begin{cor} \label{cor:normalization}
	Let $\mf{X}$ be an admissible formal scheme over $\cO_K$, let $X=\mf{X}_\eta$, and let $f\colon Y\to X$ be a morphism with $Y$ reduced for which there exists an open cover $\mf{X} = \bigcup_{i\in I} \mf{U}_i$ such that $Y\times_X \mf{U}_{i,\eta}$ is a disjoint union of finite spaces over $\mf{U}_{i,\eta}$. Then there exists a unique up to unique isomorphism morphism of formal schemes $\mf{f}\colon \mf{Y}\to\mf{X}$ with $f=\mf{f}_\eta$ and $\mf{Y}$ $\eta$-normal which locally on $\mf{X}$ is the disjoint union of finite formal schemes over $\mf{X}$.
\end{cor}

\begin{proof}
Suppose first that $Y=\coprod_{j\in J} Y_j$ with each $Y_j$ finite over $X$. In this case, we can apply Lemma~\ref{lem:normalization} to each $Y_j$ and take $\mf{Y} = \coprod_{j\in J} \mf{Y}_j$, which clearly satisfies the required property. In particular, this construction does not depend on the way $Y$ is presented as a disjoint union of finite $X$-spaces, and is compatible with passing to an open formal subscheme of $\mf{X}$. This implies that in the general case, the formal schemes $\mf{V}_i\to \mf{U}_i$ constructed in the first step will glue to provide the desired map $\mf{Y}\to \mf{X}$.
\end{proof}

We call the model $\mf{Y}\to \mf{X}$ constructed in Corollary~\ref{cor:normalization} the \emph{normalization} of $\mf{X}$ in $Y$.

\begin{rem}
The assertions of Lemma~\ref{lem:normalization} and Corollary~\ref{cor:normalization} remain valid if $K$ is stable \cite[Definition~3.6.1/1]{BGR} and with divisible value group (e.g.\@ if $K$ is algebraically closed). Indeed, under these assumptions the conclusions of \cite[Corollary 6.4.1/6]{BGR} hold (see \cite[Corollary~6.4.1/5]{BGR}).
\end{rem}

The following lemma will help us construct formal models in \S\ref{ss:tame} below.

\begin{lem} \label{lem:covering}
	Let $X$ be a quasi-paracompact and quasi-separated rigid space over $K$, and let $X=\bigcup_{i\in I} U_i$ be an open cover. Then there exists an admissible formal model $\mf{X}$ of $X$ and an open cover $\mf{X} = \bigcup_{j\in J} \mf{V}_j$ such that $\{\mf{V}_{j,\eta}\}_{j\in J}$ refines $\{U_i\}_{i\in I}$.
\end{lem}

\begin{proof}
Refining the given cover, we may assume that the $U_i$ are quasi-compact. By assumption, there is an open cover $X=\bigcup_{s\in S} X_s$ of finite type (i.e.\@,  each member intersects only finitely many others) by quasi-compact opens $X_s\subseteq X$. As $X_s$ are quasi-compact, for each $s$ there exists a finite subset $F_s \subset I$ such that $X_s \subset \bigcup_{i \in F_s} U_i$. Now, the open cover $\bigcup_{s \in S}\bigcup_{i \in F_s} (U_i \cap X_s)$ is a cover of finite type. Moreover, its members are quasi-compact, as $X$ was quasi-separated.  By \cite[8.4/5]{BoschLectures}, there exists a formal model $\mf{X}$ of $X$ and an open cover $\mf{X} = \bigcup_{s \in S}\bigcup_{i \in F_s} \mf{V}_{i,s}$ such that $(\mf{V}_{i,s})_\eta =  U_i \cap  X_s$.
\end{proof}

\subsection{A special case of the overconvergence criterion}

In order to apply our overconvergence criterion (Proposition~\ref{prop:spec-good-case}) to tame covering spaces, we show the following result.

\begin{prop} \label{prop:oc-criterion}
	Let $\mf{Y}\to\mf{X}$ be a morphism of formal schemes locally of finite type over $\cO_K$ which locally on $\mf{X}$ is the disjoint union of finite morphisms. Suppose that for every irreducible component $Z$ of $\mf{X}_k$, the base change $\mf{Y}\times_\mf{X} Z$ is locally irreducible\footnote{A topological space is locally irreducible if every point has an irreducible open neighborhood (equiv., a basis of irreducible open neighborhoods), or equivalently if it is the disjoint union of irreducible spaces.}  (e.\@g.\@ $(\mf{Y}\times_\mf{X} Z)_{\rm red}$ is normal). Then for all such $Z$, the scheme $\mf{Y}\times_{\mf{X}} Z$ is a disjoint union of finite $Z$-spaces.
\end{prop}

\begin{proof}
Let $T$ be a connected component of $\mf{Y}\times_\mf{X} Z$. By assumption, $T$ is locally irreducible, and hence irreducible. We will show that $T\to Z$ is a finite morphism, which implies the assertion. If $Z=\bigcup U_i$ is an open cover such that $\mf{Y}\times_\mf{X} U_i$ is the disjoint union of finite maps $V_{ij}\to U_i$ with $V_{ij}$ connected, then $T\times_Z U_i$ is the disjoint union of some of the $V_{ij}$. But, since $T$ is irreducible, so is the open subset $T\times_Z U_i\subseteq T$. Therefore, we must have $T\times_Z U_i = V_{ij}$ for some index $j$. Thus $T\times_Z U_i\to U_i$ is finite for all $i$, and hence $T\to Z$ is finite. 
\end{proof}

Using Proposition~\ref{prop:spec-good-case} we obtain the following corollary.

\begin{cor}\label{cor:oc-crit-special-case}
    Let $\mf{Y}\to\mf{X}$ be a morphism of formal schemes satisfying the assumptions in Proposition \ref{prop:oc-criterion} and such that $\mf{Y}_\eta\to\mf{X}_\eta$ is \'etale. Then, $\mf{Y}_\eta\to\mf{X}_\eta$ is a de Jong covering space.
\end{cor}

\begin{example} \label{ex:oc-crit-counterex}
In \cite[\S 5.3]{ALY1}, we constructed over $K$ of equal characteristic $p>0$ an example of an $\adm$-covering space $Y\to X$ which is not a de Jong covering space, and it is instructive to see why the above criterion does not apply to it. We recall the setup: $X$ is the annulus $\{|\varpi|\leqslant |x|\leqslant |\varpi|^{-1}\}$ covered by the two annuli 
\[
    U^- = \{|\varpi|\leqslant |x|\leqslant 1\}
    \quad\text{and}\quad
    U^+=\{1\leqslant |x|\leqslant |\varpi|^{-1}\},
\]
intersecting along the unit circle $U^+\cap U^-=C=\{|x|=1\}$. The restriction of $Y\to X$ to $U^\pm$ is a disjoint union of well-chosen Artin--Schreier coverings $Y^\pm_n$ ($n\in \mb{Z}$) which are split over $C$, and $\coprod_n Y_n^-$ and $\coprod_n Y^+_n$ are identified suitably over $C$ in order to make $Y$ connected.

Construct a formal model $\mf{X}$ of $X$ by gluing together the formal models
\[
    \mf{U}_+ = \Spf\left(\cO_K\langle x^+, y^+\rangle/(x^+y^+ - \varpi)\right)
    \quad \text{and} \quad 
    \mf{U}_- = \Spf\left(\cO_K\langle x^-, y^-\rangle/(x^-y^- - \varpi)\right)
\]
of $U^+$ and $U^-$ along $\mf{C} = \Spf(\cO_K\langle x^+,x^-\rangle/(x^+x^- - 1))$ identified with $D(x^\pm)\subseteq \mf{U}^\pm$. (In terms of the coordinate $x$, we have $x^+=1/x$, $x^-=x$, $y^+=\varpi x$, $y^- = \varpi/x$.)   The irreducible component $Z$ of $\mf{X}_k$ containing $\mf{C}_k \simeq \mb{G}_{m,k}$ is isomorphic to $\mb{P}^1_k$.  The restriction of the normalization $\mf{Y}$ of $\mf{X}$ in $Y\to X$ to $Z$ will look like in Figure~\ref{fig:oc-crit-counterex}: it will consist of infinitely many copies of $Z$ glued together in a complicated way over $0,\infty\in Z$, and will not be locally irreducible.
\end{example}

\begin{figure}[ht!]
	\centering
	\includegraphics[width=.3\textwidth]{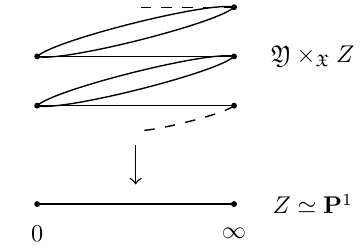}
	\caption{Example~\ref{ex:oc-crit-counterex} ($p=3$), compare with \cite[Figure~2]{ALY1}.}
	\label{fig:oc-crit-counterex}
\end{figure}

\subsection{Logarithmic Abhyankar's lemma}\label{ss:log-abhy}

The result below describes the local structure of a finite covering of a regular scheme tamely ramified along an strict normal crossings divisor \stacks{0CBN}. It is a basic consequence of the logarithmic Abhyankar's lemma (see \cite[Theorem~3.3]{Mochizuki} and \cite[Theorem~13.3.43]{GabberRameroFoundations_v7.5}), but to make the exposition more self-contained we give a direct proof using the usual Abhyankar's lemma.

In order to state the result, we first recall some basic terminology from log geometry. For a commutative monoid $P$ (written additively), we write $P\to P^{\rm gp}$ for the initial homomorphism into a group. The monoid $P$ is called \emph{integral} if $P\to P^{\rm gp}$ is injective, \emph{fine} if it is finitely generated and integral, and \emph{saturated} if it is integral and for $x\in P^{\rm gp}$ such that $nx\in P$ for some $n\geqslant 1$, we have $x\in P$. Finally, a morphism of fine and saturated monoids $\alpha\colon P\to Q$ is called \emph{Kummer} (see \cite[1.6]{IllusieFKN}) if it is injective and for every $x\in Q$ there exists an $n\geqslant 1$ such that $nx\in \alpha(P)$. Equivalently, $P=Q\cap \Lambda$ for a subgroup $\Lambda\subseteq Q^{\rm gp}$ of finite index. 

\begin{thm} \label{thm:log-abhy}
    Let $R$ be a strictly henselian regular local ring, and let $x_1, \ldots, x_r\in \mf{m}_R$ be elements whose images in $\mf{m}_R/\mf{m}^2_R$ are linearly independent. Let $X=\Spec(R)$, $D_i=V(x_i)\subseteq X$ ($i=1, \ldots, r$), and $U=D(x_1\cdots x_r) = X\setminus (D_1 \cup\ldots\cup D_r)$. Let $V\to U$ be a connected finite \'etale covering which is tamely ramified along $D_1, \ldots, D_r$, and let $f\colon Y\to X$ be the normalization of $X$ in $V$. Then there exists a Kummer homomorphism $\alpha\colon \mb{N}^r\to Q$ with $Q^{\rm gp}/\mb{Z}^r$ of order invertible in $R$ and an isomorphism of $X$-schemes 
    \[
        Y \simeq \Spec \left(R\otimes_{\mb{Z}[\mb{N}^r]} \mb{Z}[Q]\right)
    \]
    where $\mb{Z}[\mb{N}^r]\to R$ sends $e_i$ to $x_i$.
\end{thm}

Recall that if $D$ is a prime divisor on a scheme $X$ with generic point $\xi$ such that $\cO_{X, \xi}$ is a discrete valuation ring, we say that an \'etale morphism $V\to X$ is \emph{tamely ramified along $D$} if $V\otimes \op{Frac}(\cO^h_{X, \xi})$ is the disjoint union of spectra of finite extensions of $\op{Frac}(\cO^h_{X, \xi})$ which are tame with respect to the given discrete valuation.

\begin{proof}[Proof of Theorem \ref{thm:log-abhy}]
By Abhyankar's lemma \cite[Exp.\ XIII, App.\ I, Prop.\ 5.2]{SGA1}, there exists an integer $m\geqslant 1$ invertible on $X$ such that if
\[
    X_m = \Spec \left(R[y_1, \ldots, y_r]/(y_1^m - x_1, \ldots, y_r^m - x_r)\right)
\]
and $U_m = U\times_X X_m$, then $V\times_U U_m\to U_m$ extends to a finite \'etale $Y_m\to X_m$. Since $X_m$ is again the spectrum of a strictly henselian local ring, the covering $Y_m\to X_m$ admits a section. In particular, we have a morphism $U_m\to V$ of finite \'etale coverings of $U$, which is surjective because $V$ is connected. But $U_m\to U$ is a torsor for the group $\mu_m^r$ (here we write $\mu_m$ for $\mu_m(R)$, which is cyclic of order $m$ because $R$ is strictly henselian and $m$ is invertible in $R$), acting on $X_m$ and $U_m$ by $y_i\mapsto \zeta_i y_i$ for $(\zeta_1, \ldots, \zeta_r)\in \mu_m^r$. Let $H\subseteq \mu_m^r$ be the subgroup fixing the elements of the image of $\Gamma(V, \cO_V)\to \Gamma(U_m, \cO_{U_m})$, so that $V=U_m/H$. We set $Y'=X_m/H$, which is something we will explicitly compute in the next paragraph, and claim that $Y'\simeq Y$, i.e.\ that $Y'$ equals the normalization of $X$ in $V=Y'\times_X U$. For this, since $Y'$ is integral and finite over $X$, it suffices to check that $Y'$ is normal, which will also follow from the computation.

Let $\Lambda \subseteq\frac{1}{m}\mb{Z}^r$ be the orthogonal complement of $H$ with respect to the pairing 
\[ 
    \mu_m^r \times \tfrac{1}{m}\mb{Z}^r \to \mu_m, \quad (\zeta_1, \ldots, \zeta_r)\cdot \left(\tfrac{n_1}{m}, \ldots, \tfrac{n_r}{m}\right) = \zeta_1^{n_1} \cdot \ldots \cdot \zeta_r^{n_r},
\]
and set $Q = \Lambda\cap \frac{1}{m}\mb{N}^r$. Then a monomial $y_1^{n_1}\ldots y_r^{n_r} \in \Gamma(X_m, \cO_{X_m})$ is $H$-invariant if and only if $(\frac{n_1}{m}, \ldots, \frac{n_r}{m})\in Q$. Note that we have Kummer homomorphisms $\mb{N}^r\hookrightarrow Q\hookrightarrow \frac{1}{m}\mb{N}^r$, and $X_m \simeq \Spec(R\otimes_{\mb{Z}[\mb{N}^r]}\mb{Z}[\frac{1}{m}\mb{N}^r])$. Moreover, the homomorphism $\mb{Z}[Q]\to \mb{Z}[\frac{1}{m}\mb{N}^r]$ is a split injection of $\mb{Z}[\mb{N}^r]$-modules (send an element of $\frac{1}{m}\mb{N}^r$ to itself if it belongs to $Q$, otherwise send it to zero -- this gives a $\mb{Z}[Q]$-module homomorphism because $Q\hookrightarrow \frac{1}{m} \mb{N}^r$ is Kummer). We conclude from this that the natural map 
\begin{equation*}
    R\otimes_{\mathbf{Z}[\mathbf{N}^r]}\mathbf{Z}[Q]\to \Gamma(X_m,\h_{X_m})^H=\left(R\otimes_{\mathbf{Z}[\mathbf{N}^r]} \mathbf{Z}[\tfrac{1}{m}\mathbf{N}^r]\right)^H
\end{equation*}
is an isomorphism. Indeed, the surjectivity is clear from the discussion above, as is in the injectivity since we have tensored with a split injection. Therefore,
\[ 
    Y'\simeq \Spec \left(R\otimes_{\mb{Z}[\mb{N}^r]} \mb{Z}[Q]\right),
\]
 which is normal by Lemma~\ref{lem:log-abhy-normal} below.
\end{proof}

We state the lemma below separately as we shall need it again. The proof uses logarithmic geometry, but the result is also easy to show directly.

\begin{lem} \label{lem:log-abhy-normal}
    Let $R$ be a regular local ring, and let $x_1, \ldots, x_r\in \mf{m}_R$ be elements whose images in $\mf{m}_R/\mf{m}^2_R$ are linearly independent. Let $\alpha\colon \mb{N}^r\to Q$ be a Kummer homomorphism of fine and saturated monoids with $Q^{\rm gp}/\mb{Z}^r$ of order invertible in $R$ and let $\mb{N}^r\to R$ send $e_i$ to $x_i$. Then the ring $R\otimes_{\mb{Z}[\mb{N}^r]} \mb{Z}[Q]$ is normal.
\end{lem}

\begin{proof}
The log scheme $X=\Spec(\mb{N}^r\to R)$ is log regular (cf.\@ \cite[Definition 2.1]{KatoToric}). Indeed, by \cite[Proposition 7.1]{KatoToric} it suffices to check the conditions (i) and (ii) of loc.\@ cit.\@ at the closed point of $\Spec(R)$. In other words, we must show that $R/(x_1,\ldots,x_r)$ is regular and that $\dim(R)=\dim(R/(x_1,\ldots,x_r))+\mathrm{rank}(\mb{Z}^r)$ (note $\mc{M}^{\rm gp}_X/\cO^\times_X\simeq \bigoplus_{i=1}^r\underline{\mb{Z}}_{V(x_i)}$), but both these statements are obvious. Moreover, the map of log schemes 
\[ 
    \Spec(Q\to R\otimes_{\mb{Z}[\mb{N}^r]} \mb{Z}[Q]) \to \Spec(\mb{N}^r\to R)
\]
is log \'etale by \cite[Theorem 3.5]{KatoLog}. Therefore $\Spec(Q\to R\otimes_{\mb{Z}[\mb{N}^r]} \mb{Z}[Q])$ is log regular by \cite[Theorem 8.2]{KatoToric}, and so $R\otimes_{\mb{Z}[\mb{N}^r]} \mb{Z}[Q]$ is normal by \cite[Theorem 4.1]{KatoToric}.
\end{proof}

The following corollary is the crux of the argument in this section.

\begin{cor} \label{cor:log-abhy}
    Let $X$ be a regular scheme and let $D=D_1\cup\ldots\cup D_r$ be a divisor with strict normal crossings. Let $U=X\setminus D$ and let $V\to U$ be a finite \'etale morphism which is tamely ramified along $D_1, \ldots, D_r$. Let $Y\to X$ be the normalization of $X$ in $V$. Then the schemes $(Y\times_X D_i)_{\rm red}$ are normal for $i=1, \ldots, r$. 
\end{cor}

\begin{proof}
Let $y\in Y\times_X D_i$ and let $x$ be its image in $X$. Replacing $X$ with $\Spec(\cO_{X, x})$ and $Y$ with its pullback, we may assume that $X$ is local. Further, since $\cO_{X, x}\to \cO_{X, x}^{\rm sh}$ is faithfully flat, we may assume that $X$ is strictly local by \stacks{033G}. Write $X=\Spec(R)$ and $D_i=V(x_i)$ for some $x_i\in R$. We are now in a situation where Theorem~\ref{thm:log-abhy} and Lemma~\ref{lem:log-abhy-normal} both apply.

Let $Y\simeq \Spec (R\otimes_{\mb{Z}[\mb{N}^r]} \mb{Z}[Q])$ be as in Theorem~\ref{thm:log-abhy}. Our goal is to show that the ring \[
    ((R\otimes_{\mb{Z}[\mb{N}^r]} \mb{Z}[Q])\otimes_R R/(x_i))_{\rm red}
\]
is normal for $i=1, \ldots, r$. For notational convenience, we prove this for $i=r$. The ring $R'=R/(x_r)$ is again a regular local ring, and $x_1, \ldots, x_{r-1}$ are elements of the maximal ideal which are linearly independent in $\mf{m}_{R'}/\mf{m}_{R'}^2$. We claim that there exists a Kummer homomorphism of fine and saturated monoids $\alpha'\colon \mb{N}^{r-1}\to Q'$ with $(Q')^{\rm gp}/\mb{Z}^{r-1}$ of order invertible in $R'$ such that
\[ 
    ((R\otimes_{\mb{Z}[\mb{N}^r]} \mb{Z}[Q])\otimes_R R')_{\rm red} \simeq R'\otimes_{\mb{Z}[\mb{N}^{r-1}]} \mb{Z}[Q'].
\]
The right hand side of this isomorphism is normal by Lemma~\ref{lem:log-abhy-normal}, so the claim implies the required assertion. Note that the left hand side can also be written as $(R\otimes_{\mb{Z}[\mb{N}^{r}]/(e_r)} \mb{Z}[Q]/(\alpha(e_r)))_{\rm red}$.

Let us first compute $(\Spec(\mb{Z}[Q]/(\alpha(e_r))))_{\rm red}$. To this end, let $\mb{N}^{r-1}\subseteq \mb{N}^r$ be the inclusion omitting the $r$-th factor, let  $Q'\subseteq Q$ consist of elements whose positive multiple lies in $\alpha(\mb{N}^{r-1})$, and let $\alpha'\colon \mb{N}^{r-1}\to Q'$ be the induced map. Every element of $Q\setminus Q'$ has a multiple in the ideal of $\mb{Z}[Q]$ generated by $\alpha(e_r)$. Consequently, the ideal $J$ generated by $Q\setminus Q'$ in $\mb{Z}[Q]/(\alpha(e_r))$ is nilpotent. Since $\mb{Z}[Q]/(\alpha(e_r), J)=\mb{Z}[Q]/(J)\simeq \mb{Z}[Q']$ is reduced, we have $(\Spec(\mb{Z}[Q]/(\alpha(e_r)))_{\rm red}\simeq \mb{Z}[Q']$. Note that $\alpha'\colon \mb{N}^{r-1}\to Q'$ is a Kummer homomorphism of fine and saturated monoids, and the order of $(Q')^{\rm gp}/\mb{Z}^{r-1}\subseteq Q^{\rm gp}/\mb{Z}^r$ is invertible in $R$. 

Applying $R'\otimes_{\mb{Z}[\mb{N}^{r-1}]}(-)$ to the surjection with nilpotent kernel $\mb{Z}[Q]/(\alpha(e_r))\to \mb{Z}[Q']$ we get a surjection with nilpotent kernel
\[ 
    R'\otimes_{\mb{Z}[\mb{N}^{r-1}]} \mb{Z}[Q]/(\alpha(e_r)) \to R'\otimes_{\mb{Z}[\mb{N}^{r-1}]} \mb{Z}[Q'].
\]
The target is reduced (as it is normal), and hence it is the reduction of the source.
\end{proof}

\subsection{Tame coverings}
\label{ss:tame}

We now begin the setup to the main theorem of this section, which shows that under the additional hypotheses including tameness, every $\et$-covering space is a de Jong covering space. The general definition we use is due to H\"ubner \cite{Hubner}.\footnote{A more restrictive notion of tame coverings of non-archimedean curves has been studied by Berkovich \cite[\S 6.3]{BerkovichEtale}, who showed, for example, that finite \'etale coverings of the unit disc which are tame in his sense are trivial. This is not the case for the notion we use here.} 

\begin{definition} \label{tame def}
	An \'etale morphism of rigid $K$-spaces $f\colon Y\to X$ is \emph{tame} if for every $y\in Y$, the corresponding finite extension of valued fields $k(f(x))\subseteq k(y)$ is tame.
\end{definition}

The second assumption we shall put on the base space is that it has sufficiently many semistable models, as defined below. Let us call a formal scheme $\mf{X}$ locally of finite type over $\cO_K$ \emph{generalized strictly semistable (gss)} if locally $\mf{X}$ admits an \'etale map to the formal scheme $\Spf\left(\cO_K\langle x_1, \ldots, x_n\rangle/(x_1^{m_1}\cdots x_n^{m_n} - \varpi)\right)$ for some $n\geqslant 0$ and $m_1, \ldots, m_n\geqslant 0$ not all zero.

\begin{definition} \label{def:semistable-cofinal}
    Let $X$ be a smooth quasi-paracompact and quasi-separated rigid $K$-space. We say that \emph{gss formal models of $X$ are cofinal} if for every admissible formal model $\mf{X}$ of $X$ there exists a finite extension $K'$ of $K$ with valuation ring $\cO_{K'}$, a generalized strictly semistable formal model $\mf{X}'$ of $X_{K'}$ over $\cO_{K'}$, and a morphism
	\[ 
		f\colon \mf{X}' \to \mf{X}_{\cO_{K'}}
	\] 
	such that $f_\eta$ is an isomorphism.
\end{definition}

The cofinality of gss formal models for a smooth quasi-paracompact and quasi-separated rigid $K$-space condition is satisfied (with $K'=K$) if $K$ is of equal characteristic zero (see \cite{TemkinDesingularization}), or if $X$ is of dimension one. Moreover, it would follow from functorial resolution of singularities for quasi-excellent schemes of dimension $\dim(X)+1$ (again with $K'=K$).

We note that for a gss formal scheme $\mf{X}$, the complete local rings $\widehat{\cO}_{\mf{X}, x}$ are regular, and the preimages of the irreducible components of $\mf{X}_k$ in $\Spec(\widehat{\cO}_{\mf{X}, x})$ are of the form $V(x_i)$ for elements $x_1, \ldots, x_r\in \mf{m}_x$ which are linearly independent modulo $\mf{m}^2_x$. In particular, $\mf{X}$ is formally normal, and hence $\eta$-normal by Lemma~\ref{lem:eta-normal-basic-criteria}.  

\begin{lem}\label{lem:normalization-tame} 
    Let $\mf{X}=\Spf(A)$ be a gss formal scheme over $\cO_K$ and let $f\colon Y\to \mf{X}_\eta$ be a finite \'etale and tame map of rigid $K$-spaces. Then, writing $Y=\Spa(B_K)$, the morphism of schemes $\Spec(B_K)\to \Spec(A)$ is tamely ramified along the irreducible components of $\Spec(A_k)$.
\end{lem}

\begin{proof}
Let $\xi_1, \ldots, \xi_r$ be the generic points of the irreducible components of $\mf{X}_k$. The complete local rings $V_i=\widehat{\cO}_{\mc{X}, \xi_i}$ are discrete valuation rings and $\varpi$ is in the maximal ideal of $V_i$. We let $K_i = V_i[1/\varpi]=\op{Frac}(V_i)$. We have maps of Huber pairs $(A_K, A)\to (K_i, V_i)$ and pull-back squares
\[ 
	\xymatrix{
		\Spa(B_K\otimes_A K_i) \ar[r] \ar[d] & Y \ar[d] \\
		\Spa(K_i) \ar[r] & X .
	}
\]
Thus the left map is a tame \'etale map of adic spaces \cite[\S 3]{Hubner} and so $\Spec(B_K\otimes_A K_i)\to \Spec(V_i)$ is tamely ramified along the closed point of $\Spec(V_i)$. Then the same holds if we replace $V_i$ with the henselization $\cO_{\mc{X},\xi_i}^h$ (because $\cO_{\mc{X},\xi_i}^h$ and $V_i = \widehat{\cO}_{\mc{X},\xi_i}^h$ have the same inertia group), and hence $\Spec(B_K)\to \Spec(A)$ is tamely ramified at $\xi_i$.
\end{proof}

We now come to our main theorem.

\begin{thm} \label{thm:tame-comparison}
	Let $X$ be a smooth quasi-paracompact and quasi-separated rigid-analytic space and suppose that gss formal models of $X$ are cofinal. Let $Y\to X$ be an $\adm$-covering space. If $Y\to X$ is tame, then it is a de Jong covering space.
\end{thm}

\begin{proof}
To show the assertion, it suffices to prove that $Y_{K'}\to X_{K'}$ is a de Jong covering space for some finite extension of $K'$ of $K$. Indeed, in the case of a separable extension, this follows from \cite[Lemma 2.3]{deJongFundamental}), and in the purely inseparable case this is simple since $\Spa(K')\to\Spa(K)$ is a universal homeomorphism. By assumption, there exists an open cover $X=\bigcup_{i\in I} U_i$ such that $Y|_{U_i}$ is the disjoint union of finite \'etale covers of $U_i$. By Lemma~\ref{lem:covering}, we may assume that $U_i = \mf{U}_{i,\eta}$ for an open cover $\mf{X}=\bigcup_{i\in I}\mf{U}_i$ of some admissible formal model $\mf{X}$ of $X$. Replacing $K$ with a finite extension and performing an admissible blowup, we may assume that $\mf{X}$ is generalized strictly semistable, and in particular $\eta$-normal. 

Let $\mf{Y}\to \mf{X}$ be the unique model of $Y\to X$ with $\mf{Y}$ $\eta$-normal and which is locally on $\mf{X}$ the disjoint union of finite morphisms, provided by Corollary~\ref{cor:normalization}. Let $D_1, \ldots, D_r$ be the irreducible components of $(\mf{X}_k)_{\rm red}$. Thanks to Proposition~\ref{prop:oc-criterion}, it suffices to show that the schemes $(\mf{Y}\times_{\mf{X}} D_i)_{\rm red}$ are normal. 

This assertion being local, we may assume that $\mf{X}=\Spf(A)$ is affine and that $Y$ is the disjoint union of finite \'etale covers of $X=\Spa(A_K)$. Working one connected component of $Y$ at a time, we may assume that $Y\to X$ is finite \'etale, and therefore $\mf{Y}=\op{Spf}(B)$ where $B=B_K^\circ$, $A\to B$ is finite, and $A_K\to B_K$ is finite \'etale. 

Since $\Spf(A)$ and $\Spec(A)$ have the same completed local rings at closed points, and the same special fiber, the scheme $\Spec(A)$ is regular and the divisor $V(\varpi)_{\rm red} = D_1\cup \ldots \cup D_r$ has strict normal crossings. By Lemma~\ref{lem:normalization-tame}, the map of schemes $\Spec(B_K)\to \Spec(A)$ is tamely ramified along the $D_i$. Therefore, by Corollary~\ref{cor:log-abhy}, the schemes $T_i=(\Spec(B)\times_{\Spec(A)} D_i)_{\rm red}$ are normal. Since the $D_i$ reside in the special fiber, we have $T_i = \mf{Y}\times_{\mf{X}} D_i$, and the theorem is proved.
\end{proof}

\begin{rem}
In the language of logarithmic geometry, the proof of Theorem~\ref{thm:tame-comparison} gives the following result. If $\mf{X}$ is a gss formal scheme over $\cO_K$ with rigid generic fiber $X=\mf{X}_\eta$ and if $Y\to X$ is a tame finite \'etale covering, then $Y$ extends uniquely to a Kummer \'etale morphism $\mf{Y}\to\mf{X}$ where $\mf{Y}$ and $\mf{X}$ are endowed with the natural log structures induced by $\mf{Y}_k$ and $\mf{X}_k$. Thus the category $\FEt_X$ is equivalent to the category of finite and Kummer \'etale coverings of $\mf{X}$ treated as a log formal scheme.
\end{rem}

We finally extend our result to the fully generality by showing that every tame $\et$-covering space is a de Jong covering space. We shall essentially use a bootstrap argument to reduce to Theorem \ref{thm:tame-comparison}.

\begin{cor} \label{cor:tame-comparison}
    Let $X$ be a smooth quasi-paracompact and quasi-separated rigid $K$-space and suppose that the following conditions hold:
    \begin{enumerate}[(1)]
        \item $X$ has cofinal semistable models,
        \item every finite \'etale covering of an affinoid open of $X$ has cofinal semistable models.
    \end{enumerate}
    Then every tame $\et$-covering space of $X$ is a de Jong covering space.
\end{cor}

\begin{proof}
Let $Y\to X$ be a tame $\et$-covering space, and let $U\to X$ be an \'etale cover such that $Y_U$ is the disjoint union of finite \'etale coverings. By \cite[Proposition~2.8.2]{ALY1}, the cover $\{U\to X\}$ admits as a refinement a cover of the form $\{V_{ij}\to X\}_{i\in I,j\in J_i}$ such that for all $i\in I$ and $j\in J_i$ one has a factorization
\[
     V_{ij}\to W_i\to U_i\to X 
\]
where $\{U_i\to X\}_{i\in I}$ forms an open cover, the maps $W_i\to U_i$ are finite \'etale, and for every $i\in I$, the family $\{V_{ij}\to W_i\}_{j\in I_j}$ is a finite affinoid open cover. Moreover, refining further we may assume that $\{U_i\to X\}_{i\in I}$ is an affinoid cover of finite type. Now, each $Y_{W_i}\to W_i$ is a tame $\adm$-covering space, and by assumption (2) and Theorem~\ref{thm:tame-comparison}, we see that it is a de Jong covering space. Therefore, by \cite[Lemma~2.3]{deJongFundamental} also each $Y_{U_i}\to U_i$ is a de Jong covering space. Thus $Y\to X$ is an $\adm$-covering space, and we conclude using assumption (1) and Theorem~\ref{thm:tame-comparison}.
\end{proof}

As remarked before all smooth quasi-paracompact and quasi-separated rigid $K$-varieties, where $K$ has equal characteristic zero, have a cofinal gss formal models. Moreover, since tameness is automatic when the residue field $k$ is of characteristic zero, we deduce the following result.

\begin{cor} \label{cor:tame-char0} 
    Let $K$ be a complete discrete valuation field of residue characteristic zero, and let $X$ be a smooth quasi-paracompact and quasi-separated rigid-analytic space over $K$. Then every $\et$-covering space of $X$ is a de Jong covering space.
\end{cor}

\appendix

\section{\texorpdfstring{$\eta$-normality}{eta-normality}}
\label{app:eta-normal}

The rigid generic fiber of a connected admissible formal scheme over $\cO_K$ may be disconnected. In order to control its connectedness, we introduce in this appendix the condition of $\eta$-normality, i.e.\ being ``integrally closed in the rigid generic fiber.'' To be able to apply it to questions concerning covering spaces, we need to show that if $\mf{X}'\to \mf{X}$ is \'etale and $\mf{X}$ is $\eta$-normal, then so is $\mf{X}'$ (Proposition~\ref{prop:etale-eta-normal}). To this end, we provide a version of Serre's criterion ``normal $=(S_2)+(R_1)$'' (Proposition~\ref{prop:serre-crit}). The reader should be aware that many of the proofs below are considerably easier in the situation when $K$ is discretely valued.

\subsection{\texorpdfstring{$\eta$-normality}{eta-normality}} 

We start with a definition of $\eta$-normality (cf.\@ \cite[Remark 7.4.2]{deJongCrystal}).

\begin{definition} \label{def:eta-normal}
	An $\cO_K$-algebra $A$ topologically of finite type over $\cO_K$ is called \emph{$\eta$-normal} if it is $\cO_K$-torsion free and integrally closed in $A_K$. A formal scheme $\mf{X}$ locally of finite type over $\cO_K$ is called \emph{$\eta$-normal} if for every affine open $\Spf(A)\subseteq \mf{X}$, the algebra $A$ is $\eta$-normal.
\end{definition}

In other words, $A$ is $\eta$-normal if $A=A_K^\circ$. In particular, $A$ is reduced since if $f$ is nilpotent, then $K\cdot f\subseteq A_K^\circ$ and $A_K^\circ$ is not bounded which contradicts that it is equal to $A$. One can rephrase the condition that $\mf{X}$ is $\eta$-normal by saying that the map of sheaves $\cO_\mf{X}\to \op{sp}_{\mf{X},*}\, \cO_{\mf{X}_\eta}^+$ is an isomorphism. In particular, the property of being $\eta$-normal can be checked on a single affine open cover of $\mf{X}$.

The following lemma gives common conditions which guarantee $\eta$-normality.

\begin{lem} \label{lem:eta-normal-basic-criteria}
    Suppose that one of the following conditions holds:
    \begin{enumerate}[(a)]
        \item $\mf{X}_k$ is reduced,
        \item $\mf{X}$ is formally normal (i.e.\@ $\h_{\mf{X},x}$ is normal for all $x$ in $\mf{X})$.
    \end{enumerate}
    Then $\mf{X}$ is $\eta$-normal.
\end{lem}

\begin{proof}
For (a) we may assume that $\mf{X}$ is affine, in which case the claim follows from \cite[Proposition 1.1]{BLRIV}. For (b) it suffices to show that if $\Spf(A)$ is formally normal and connected, then $A$ is a normal domain. Since $\Spec(A)$ is connected and topologically Noetherian, we know by \stacks{0357} that it suffices to show that $\Spec(A)$ is a normal scheme. But, by \stacks{030B} it suffices to show that for every maximal ideal $\mf{m}$ of $A$ that $A_\mf{m}$ is a normal domain. Note though that since $A$ is $\varpi$-adically complete, that $\varpi$ is contained in the Jacobson radical of $A$ (see \cite[Chapter 0, Lemma 7.2.13]{FujiwaraKato}), and so $\mf{m}\supseteq \sqrt{(\varpi)}A=\mf{m}_KA$. Thus, we see that $\mf{m}$ corresponds to a point $x$ of $\Spf(A)$. Moreover, we have a natural map $A_\mf{m}\to \h_{\mf{X},x}$ obtained as the colimit of the natural maps $A[f^{-1}]\to A\langle f^{-1}\rangle$. By Gabber's theorem (see \cite[Lemma 8.2/2]{BoschLectures}) each map $A[f^{-1}]\to A\langle f^{-1}\rangle$ is flat, and thus is the map $A_\mf{m}\to\h_{\mf{X},x}$. This map is local (cf.\@ \cite[Remark 7.2/1]{BoschLectures}) and thus faithfully flat by \stacks{00HQ}. Thus, $A_\mf{m}$ is normal by \stacks{033G}.
\end{proof}

For us, the main use for $\eta$-normality is the fact that connectedness is preserved when passing to the rigid generic fiber.

\begin{lem}\label{lem:connected generic fiber} 
    Let $\mf{X}$ be formal scheme locally of finite type over $\cO_K$ which is $\eta$-normal. If $\mf{X}$ is connected, then so is $\mf{X}_\eta$. 
\end{lem}

\begin{proof}
By assumption, the map $\cO_{\mf{X}} \to \op{sp}_{\mf{X},*}\, \cO_{\mf{X}_\eta}^+$ is an isomorphism. Since every idempotent $\Gamma(\mf{X}_\eta, \cO_{\mf{X}_\eta})$ belongs to $\Gamma(\mf{X}_\eta, \cO_{\mf{X}_\eta}^+)=\Gamma(\mf{X}, \cO_{\mf{X}})$, we deduce that $\Gamma(\mf{X}_\eta, \cO_{\mf{X}_\eta})$ has no nontrivial idempotents, so $\mf{X}_\eta$ is connected.
\end{proof}

\subsection{\texorpdfstring{Conditions $(S'_1)$ and $(S'_2)$}{Conditions S'1 and S'2}} 

In the absence of a robust theory of depth and embedded components in the non-Noetherian setting, we use the following condition as a substitute for the condition $(S_1)$. 

\begin{definition} \label{def:s1-prime}
    Let $X$ be a locally topologically Noetherian scheme. We say that $X$ satisfies \emph{condition $(S'_1)$} if every dense open subscheme $U\subseteq X$ is schematically dense (i.e.\@, $\cO_X\to j_* \cO_U$ is injective, where $j\colon U\to X$ is the inclusion).
\end{definition}

\begin{lem} \label{lem:s1-prime-conditions}
    Let $X$ be a locally topologically Noetherian scheme.
    \begin{enumerate}[(a)]
        \item If $X=\Spec(A)$ is affine, then $X$ satisfies $(S_1')$ if and only if every zerodivisor of $A$ belongs to at least one minimal prime of $A$. 
        \item The scheme $X$ satisfies condition $(S_1')$ if and only if for one (equiv.\@ for all) open covers $\{U_i\}$ of $X$ each $U_i$ satisfies condition $(S_1')$. 
        \item If $X$ is locally Noetherian, then $X$ satisfies $(S'_1)$ if and only if it has no embedded points, or equivalently if it satisfies condition $(S_1)$.
    \end{enumerate}
\end{lem}

\begin{proof}
(a) Let us first note that an open of the form $D(g)$, $g\in A$ is dense in $X=\Spec(A)$ if and only if $g$ belongs to no minimal prime of $A$. To see that $(S'_1)$ implies the given condition, take $U=D(g)$ and note that the equation $gf=0$ implies that $f|_U=0$, and hence $f=0$. 

Conversely, assume the condition holds and let $U\subseteq X$ be a dense open. We first claim that $U$ contains a dense open of the form $D(g)$ for some $g\in A$. Indeed, writing $X\setminus U=V(I)$ for some ideal $I\subseteq A$, and denoting by $\mf{p}_1, \ldots, \mf{p}_r$ the minimal primes of $A$, the condition that $U$ is dense means that $I\not\subseteq \mf{p}_i$ for all $i$. By prime avoidance \stacks{00DS} there exists a $g\in I$ with $g\notin \mf{p}_i$ for all $i$. Then $D(g)\subseteq U$ and $D(g)$ is dense. Thus, we may assume that $U=D(g)$. 

Now, the inclusion $j\colon U\to X$ is affine, and hence $\cO_X\to j_*\cO_U$ is injective if and only if $A\to A[\frac 1 g]$ is injective. Its kernel consists of elements $f\in A$ such that $g^n f =0$ for some $n\geqslant 1$. Since the powers $g^n$ belong to no minimal prime of $A$, the given condition implies that $f=0$. 

(b) If $X=\bigcup U_i$ and the $U_i$ satisfy $(S'_1)$, then so does $X$. For the converse, it suffices to show that if $X$ satisfies $(S'_1)$, then every open $X'\subseteq X$ satisfies $(S'_1)$. If $U'\subseteq X'$ is a dense open, then let $U$ be the union of all opens of $X$ whose intersection with $X'$ equals $U'$. Then $U$ is a dense open of $X$, so $\cO_X\to j_*\cO_U$ is injective, and so is its restriction to $X'$.

(c) Combine \stacks{083P} and \stacks{0346}. 
\end{proof}

\begin{lem} \label{lem:etale-s1-prime}
    Let $Y\to X$ be an \'etale morphism of locally topologically Noetherian schemes\footnote{N.B.\ If $Y\to X$ is \'etale and $X$ is locally topologically Noetherian, then so is $Y$, by \cite[Lemma~6.6.10(3)]{BhattScholze}.} If $X$ satisfies $(S_1')$, then so does $Y$.
\end{lem}

\begin{proof}
Since the condition is local on $Y$, we may assume that $X$ and $Y$ are affine and $f$ is standard \'etale, i.e.
\[ 
    X=\Spec(A), \qquad Y=\Spec(B), \qquad B=(A[T]/(f))_g
\]
where $f\in A[T]$ is a monic polynomial whose derivative $f'$ is invertible in $B$ (see \stacks{02GT}). Let $C=A[T]/(f)$ and let $Z=\Spec(C)$, so that $Y$ is an open subset of $Z$. Since $f$ is monic, $C$ is a free $A$-module of rank $\deg(f)$, and hence the morphism $\varphi\colon Z\to X$ is finite and locally free and therefore locally of finite presentation (see \stacks{02KB}).

Since $Y$ is an open in $Z$, it suffices to show that $Z$ satisfies $(S_1')$. To this end, suppose that $W\subseteq Z$ is a dense open and $x\in C$ is an element whose restriction to $W$ is zero. We have to show that $x=0$. Let $U=X\setminus \varphi(Z\setminus W)$; this is an open subset because $\varphi$ is a closed map, and $\varphi^{-1}(U)$ is contained in $W$. We claim that $\varphi^{-1}(U)$ is also dense in $Z$; otherwise there exists a non-empty open $V\subseteq Z$ with $\varphi(V)$ contained in $\varphi(Z\setminus W)$. But $\varphi(V)$ is open (since $\varphi$ is flat and finitely presented, see \stacks{01UA}) and so contains a generic point $\xi$. That said, the fiber $\varphi^{-1}(\xi)$ consists entirely of generic points (cf.\@ \stacks{0ECG}) which would imply that $Z\setminus W$ contains a generic point, which is impossible.

Writing $x=\sum_{i<\deg(f)} a_i T^i$ with $a_i\in A$, if $x$ vanishes on $W$, it vanishes on $\varphi^{-1}(U)$, and then the coefficients $a_i$ vanish on the dense open $U\subseteq X$. Since $X$ satisfies $(S'_1)$, we must have $a_i=0$ for all $i$, and hence $x=0$, as desired.
\end{proof}

\begin{prop}\label{prop:s1'-diff-pu}
    Let $A$ be an admissible $\h_K$-algebra. Then the following conditions are equivalent:
    \begin{enumerate}[(a)]
        \item For some pseudouniformizer $\varpi$, the scheme $\Spec(A/\varpi A)$ satisfies $(S'_1)$.
        \item For every pseudouniformizer $\varpi$, the scheme $\Spec(A/\varpi A)$ satisfies $(S'_1)$.
    \end{enumerate}
\end{prop}

\begin{proof}
It suffices to show the following two claims:
\begin{enumerate}[(1)]
    \item If $\Spec(A/\varpi' A)$ and $\Spec(A/\varpi'' A)$ both satisfy $(S_1')$, then so does $\Spec(A/\varpi A)$, where $\varpi = \varpi'\varpi''$.
    \item If $\Spec(A/\varpi A)$ satisfies $(S_1')$ and $\varpi'$ divides $\varpi$, then $\Spec(A/\varpi' A)$ satisfies $(S_1')$.
\end{enumerate}
For the first claim, suppose that $g\in A/\varpi A$ belongs to no minimal prime and $f\in A/\varpi A$ is such that $gf=0$. Reducing modulo $\varpi'$ and using the $(S1')$ property of $\Spec(A/\varpi' A)$, we see that $f \in \varpi' A/\varpi A$, say $f = \varpi' f'$. Then $\varpi' f' g = 0$ implies (because $A$ is flat over $\cO_K$) that $f' g = \varpi'' h$ for some $h$\footnote{Indeed, $\varpi'f'g=0$ means that $\varpi'f'g=\varpi h$ in $A$. This implies that $\varpi'(f'g-\varpi''h)=0$ in $A$ which, since $A$ has no $\h_K$-torsion, implies that $f'g-\varpi''h=0$ and so $f'g=\varpi''h$ as desired.}. Reducing modulo $\varpi''$ we see that $f'g=0$ in $A/\varpi''A$ and so using the $(S_1')$ property of $\Spec(A/\varpi'')$, we see that $f'=0$ in $A/\varpi'A$ and so $f' = \varpi'' f''$ in $A/\varpi A$, so $f= \varpi'f'=\varpi' \varpi'' f'' = 0$ as desired.

For the second claim, suppose that $g'\in A/\varpi' A$ belongs to no minimal prime and $f'\in A/\varpi' A$ is such that $g'f'=0$. Let $f$ and $g$ be lifts to $A/\varpi A$, so that $gf = \varpi' h$ for some $h$. Write $\varpi= \varpi'\varpi''$, so $f'' := \varpi'' f$ has the property that $gf'' = \varpi''gf =\varpi''\varpi' h =0$. By the $(S_1')$ property of $\Spec(A/\varpi A)$, we thus have $f''=\varpi''f =0$. Since $A$ is flat over $\cO_K$, we must have $f \in \varpi' A$, so $f'=0$. 
\end{proof}

\begin{rem}
We do not know if the conditions in Proposition~\ref{prop:s1'-diff-pu} are further equivalent to the condition that $\Spec(A_k)$ satisfies $(S'_1)$ (or, equivalently, $(S_1)$) in the case when $K$ is not discretely valued. 
\end{rem}

\begin{definition} \label{def:s2-prime}
    Let $\mf{X}$ be a formal scheme locally of finite type over $\cO_K$. We say that $\mf{X}$ satisfies condition $(S'_2)$ if it is flat over $\h_K$ and for every (equiv.\@ any) pseudouniformizer $\varpi$, the scheme $\mf{X}\otimes (\cO_K/\varpi)$ satisfies condition $(S'_1)$.
\end{definition}

The following result follows from Lemma~\ref{lem:etale-s1-prime} and \cite[Chapter I, Proposition 5.3.11]{FujiwaraKato}.

\begin{lem} \label{prop:etale-s2-prime}
    Let $\mf{Y}\to \mf{X}$ be an \'etale morphism of formal schemes locally of finite type over $\cO_K$. If $\mf{X}$ satisfies $(S'_2)$, then so does $\mf{Y}$.
\end{lem}

\subsection{\texorpdfstring{Condition $(R'_1)$}{Condition R'1}}

In our relative situation over $\h_K$ the replacement for the usual condition $(R_1)$ for a formal scheme $\mf{X}$ is a statement about the local rings of $\mf{X}$ at the generic points $\zeta$ of $\mf{X}_k$. Since $\zeta$ is codimension zero in $\mf{X}_k$ it is intuitively of codimension one in $\mf{X}$ and so we would then like to claim that $\h_{\mf{X},\zeta}$ is a DVR. Since $\mf{X}$ is not necessarily Noetherian, we instead require that $\h_{\mf{X},\zeta}$ is a valuation ring.

To this end, for a formal scheme $\mf{X}$ locally of finite type over $\h_K$ let us denote by $I(\mf{X})$ the set of generic points of the scheme $\mf{X}_k$ or, equivalently, the set of generic points of the locally spectral space $|\mf{X}|$.

\begin{definition} \label{def:r1-prime}
    Let $\mf{X}$ be an admissible formal scheme over $\cO_K$. We say that $\mf{X}$ satisfies condition $(R'_1)$ if for every $\zeta$ in $I(\mf{X})$, the local ring $\cO_{\mf{X}, \zeta}$ is a valuation ring.
\end{definition}

Assume that $\mf{X}$ is quasi-compact and set $X=\mf{X}_\eta$. It will be useful in the proof of Lemma \ref{lem:etale-r1-prime} and Proposition \ref{prop:serre-crit} to relate the ring $\h_{\mf{X},\zeta}$ for $\zeta$ in $I(\mf{X})$ to local rings of $\h_X^+$. To this end, following \cite[Chapter II, Definition 11.1.2]{FujiwaraKato} let us denote by $D(\mf{X})$ the set of divisorial points of $X$ which are reisidually finite over $\mf{X}$. By \cite[Chapter II, Proposition 11.1.3]{FujiwaraKato} and \cite[Chapter II, Theorem 11.2.1]{FujiwaraKato} one has that the set $D(\mf{X})$ is a finite set of maximal points, and that the equality
\begin{equation*}
    D(\mf{X})=\bigcup_{\zeta\in I(\mf{X})} \mathrm{sp}_{\mf{X}}^{-1}(\zeta)
\end{equation*}
holds. Moreover, one has by \cite[Chapter II, Theorem 11.2.1]{FujiwaraKato} and \cite[Proposition 6.3/3]{BGR} that 
\begin{equation}\label{eq:spectral-seminorm-formula}
    \|f\|_\mathrm{sup}=\sup_{x\in D(\mf{X})}|f(x)|
\end{equation}
where $\|\cdot\|_\mathrm{sup}$ is the supremum semi-norm as in \cite[\S6.2.1]{BGR}.

Now, for $x$ in $D(\mf{X})$ with specialization $\zeta=\mathrm{sp}_\mf{X}(x)$ one has a natural map of topological rings $\h_{\mf{X},\zeta}\to \h_{X,x}^+$ (this map exists even if $\mf{X}$ is only admissible over $\h_K$). Indeed, since $(X,\h_X^+)$ is isomorphic to the inverse limit $\varprojlim\, (\mf{X}',\h_{\mf{X}'})$ where $\mf{X}'$ travels over the admissible blowups of $\mf{X}$ one sees that there is a canonical identification
\begin{equation}\label{eq:stalk-limit}
    \h_{X,x}^+=\varinjlim \h_{\mf{X}',\op{sp}_{\mf{X}'}(x)}
\end{equation}
(see \cite[Chapter 0, Proposition 4.1.10]{FujiwaraKato}) from where the map $\h_{\mf{X},\zeta}\to\h_{X,x}^+$ is clear. Thus, if $k(x)^+$ denotes the (uncompleted) residue valuation ring, one obtains a map $\h_{X,\zeta}\to k(x)^+$.

\begin{lem} \label{lem:r1'-equivalence}
    Let $\mf{X}$ be an admissible formal scheme over $\h_K$, let $x\in D(\mf{X})$ with image $\zeta\in I(\mf{X})$. Then, the following are equivalent:
    \begin{enumerate}[(a)]
        \item the ring $\h_{\mf{X},\zeta}$ is a valuation ring,
        \item the map $\h_{\mf{X},\zeta}\to k(x)^+$ is an isomorphism.
    \end{enumerate}
    In particular, under these conditions, the ring $\h_{\mf{X},\zeta}$ is a height one valuation ring which is $\varpi$-adically separated and Henselian.
\end{lem}

\begin{proof} 
Clearly (b) implies (a). Moreover, the last statement follows then immediately from the fact that $x$ is a maximal point and \cite[Chapter II, Corollary 3.2.8]{FujiwaraKato}. To see that (a) implies (b), let us begin by observing that for any any admissible blowup $\mf{X}'$, the map $\h_{\mf{X},\zeta}\to \h_{\mf{X}',\zeta'}$ is an isomorphism where $\zeta'=\op{sp}_{\mf{X}'}(x)$. Indeed, suppose that $\mf{X}'\to\mf{X}$ is the blowup at an admissible ideal sheaf $\mc{J}$. Since $\mc{J}$ is finitely generated we know that $\mc{J}_\zeta\subseteq \h_{\mf{X},\zeta}$ is a finitely generated ideal which, since $\h_{\mf{X},\zeta}$ is valuation ring, must be invertible. We may then find an open $\mf{U}$ of $\mf{X}$ containing $\zeta$ such that $\mc{J}|_{\mf{U}}$ is invertible. But, $\mf{X}'_{\mf{U}}\to\mf{U}$ coincides with the blowup of $\mf{U}$ along $\mc{J}|_\mf{U}$ (see \cite[Chapter II, Proposition 1.1.8]{FujiwaraKato}) but since $\mc{J}|_\mf{U}$ is invertible this implies that $\mf{X}'_\mf{U}\to\mf{U}$ is an isomorphism and so evidently $\h_{\mf{X},\zeta}\to\h_{\mf{X}',\zeta'}$ is an isomorphism as desired.

Combining this with \eqref{eq:stalk-limit} we deduce that the map $\h_{\mf{X},\zeta}\to\h_{X,x}^+$ is an isomorphism. To show then that the map $\h_{\mf{X},\zeta}\to k(x)^+$ is an isomorphism note that for any admissible blowup $\mf{X}'$ of $\mf{X}$ the map $\h_{\mf{X},\zeta}\to\h_{\mf{X}',\zeta'}$ is an isomorphism and commutes with the natural maps to $k(x)^+$. Thus, we may replace $\h_{\mf{X},\zeta}$ with any $\h_{\mf{X}',\zeta'}$. But, we in particular may take $\mf{X}'$ as in \cite[Chapter II, Proposition 11.2.8]{FujiwaraKato}. Since, the kernel of $\h_{\mf{X},\zeta}\simeq \h_{X,x}^+\to k(x)^+$ is precisely those $f$ with $|f(x)|=0$ we will be done if we can show that for $f$ in $\h_{\mf{X},\zeta}$ one has that $|f(x)|=0$ implies that $f=0$. But, we may as well assume that $f$ is in $\h(\mf{U}_\eta)$ for some affine irreducible neighborhood $\mf{U}$ of $\zeta$. But, we then see that by our choice of model $D(\mf{U})=\op{sp}_\mf{X}^{-1}(\zeta)=\{x\}$. Thus, by \eqref{eq:spectral-seminorm-formula} we deduce that $\|f\|_{\mathrm{sup}}=0$ (where this supremum is taken relative to $\mf{U}_\eta$) and thus by \cite[Proposition 3.1/10]{BoschLectures} we deduce that $f$ is a nilpotent element, and since $\h_{\mf{X},\zeta}$ is a domain, has trivial image in $\h_{\mf{X},\zeta}$.
\end{proof}

We now wish to establish that the $(R_1')$ condition ascends through \'etale maps of formal schemes. To do this, we first need the following lemma.

\begin{lem} \label{lem:finite-etale-maps}
    Let $f\colon \Spf(B)\to \Spf(A)$ be a finite \'etale morphism of admissible formal schemes over $\h_K$. Then, the map $A\to B$ is finite \'etale.
\end{lem}

\begin{proof} 
We claim that the ring map $A\to B$ is weakly \'etale (see \cite[Definition 1.2]{BhattScholze}). Indeed, we must show that $A\to B$ is flat, and that $B\to B\otimes_A B$ are flat. But, note that since $A\to B$ is finite, $B\otimes_A B$ is finite over $B$, and so complete. From this and \cite[Chapter I, Proposition 4.8.1]{FujiwaraKato} we see that it suffices to check that $A/\varpi^m A\to B/\varpi^m B$ is weakly etale for all $m\geqslant 0$, but this is clear since this map is \'etale. Note though that since $A\to B$ is finite, $B$ is actually finitely presented as an $A$-module (see \cite[Theorem 7.2/4]{BoschLectures}) and so $A\to B$ is finitely presented (see \stacks{0D46}). Since $A\to B$ is both finitely presented and weakly \'etale, it is \'etale (see \cite[Theorem 1.3]{BhattScholze}).
\end{proof}

\begin{lem}\label{lem:etale-r1-prime}
    Let $f\colon \mf{Y}\to \mf{X}$ be an \'etale morphism of admissible formal schemes over $\cO_K$. If $\mf{X}$ satisfies $(R'_1)$, then so does $\mf{Y}$. 
\end{lem}

\begin{proof}
Let $\xi$ be a generic point of $\mf{Y}$ and let $\zeta=f(\xi)$. We first show that the map $\h_{\mf{X},\zeta}\to\h_{\mf{Y},\xi}$ is finite \'etale. Note that $\zeta$ is a generic point of $\mf{X}$. Indeed, it suffices to show that $f_1(\xi)$ is a generic point where $f_1$ is the reduction of $f$ modulo $(\varpi)$, but this is clear, as $f_1$ is an open morphism of topologically Noetherian schemes.  Note then, that since $f_1$ is an \'etale morphism of schemes, there exists open neighborhoods $U$ (resp.\@ $V$) of $\zeta$ (resp.\@ $\xi$) such that $f(V)\subseteq U$, $f_1|_{V}\colon V\to U$ is finite \'etale, and $f_1^{-1}(\zeta)\cap V=\{\xi\}$. Indeed, this follows from the fact that the local ring at $\zeta$ has reduced subscheme which is a field, and so (by \stacks{04DZ}) all \'etale algebras are finite product of finite \'etale algebras. From this and \cite[Chapter I, Proposition 4.2.3]{FujiwaraKato} we see that we may find affine neighborhoods $\Spf(A)$ (resp.\@ $\Spf(B)$) of $\zeta$ (resp.\@ $\xi$) such that $f(\Spf(B))\subseteq \Spf(A)$ and $f|_{\Spf(B)}\colon \Spf(B)\to \Spf(A)$ is finite \'etale and $f^{-1}(\zeta)\cap \Spf(B)=\{\xi\}$. Replace $f\colon \mf{Y}\to\mf{X}$ by $f|_{\Spf(B)}$. By Lemma \ref{lem:finite-etale-maps} we then know that $A\to B$ is finite \'etale. Now, since $f$ is finite and $f^{-1}(\zeta)=\{\xi\}$ the neighborhoods of the form $f^{-1}(D(f))$ for $D(f)$ a neighborhood of $\zeta$ are cofinal in the neighborhoods of $\xi$. Thus, $\h_{\mf{Y},\xi}=\h_{\mf{X},\zeta}\otimes_A B$, and since $A\to B$ is finite \'etale we deduce that $\h_{\mf{X},\zeta}\to\h_{\mf{Y},\xi}$ is finite \'etale as desired. In particular, since $\h_{\mf{X},\zeta}$ is a normal domain, so is $\h_{\mf{Y},\xi}$ (see \stacks{025P}).

It remains to show that if $V\to V'$ is a finite \'etale local homomorphism between local normal domains and $V$ is a Henselian valuation ring then $V'$ is valuation ring. Set $K$ (resp.\@ $K'$) to be the fraction field of $V$ (resp.\@ $V'$). Take a nonzero $f\in K'$, and suppose that ${\rm Nm}_{K'/K}(f)$ belongs to $V$. By Hensel's lemma (see e.g.\ \cite[Corollary~6.5]{Neukirch}), the minimal polynomial of $f$ over $K$ has coefficients in $V$, and hence $f$ is integral over $V$ and hence belongs to $V'$. If ${\rm Nm}_{K'/K'}(f)$ does not belong to $V$, then ${\rm Nm}_{K'/K}(f^{-1})$ does, and so either $f$ or $f^{-1}$ belongs to $V'$.
\end{proof}

\subsection{\texorpdfstring{Serre's criterion for $\eta$-normality}{Serre's criterion for eta-normality}}  

Serre's criterion states that a locally Noetherian scheme is normal if and only if it is both $(S_2)$ and $(R_1)$. In our situation of admissible formal schemes over $\cO_K$, we have the following analog (compare with condition $(N)$ in \cite[\S 6.1]{RaynaudPicard}). 

\begin{prop}[Serre's criterion for $\eta$-normality] \label{prop:serre-crit}
    Let $\mf{X}$ be an admissible formal scheme over $\cO_K$. The following are equivalent:
    \begin{enumerate}[(a)]
        \item $\mf{X}$ is $\eta$-normal,
        \item $\mf{X}$ satisfies $(S'_2)$ and $(R'_1)$.
    \end{enumerate}
\end{prop}

\begin{proof}
Since being $\eta$-normal and the conditions in (b) are Zariski local on $\mf{X}$, we may assume that $\mf{X}=\Spf(A)$ where $A$ is an admissible $\h_K$-algebra. Set $X$ to be $\mf{X}_\eta=\Spa(A_K)$. 
    
Suppose first that condition (a) holds. To show condition $(S'_2)$ it suffices to show that if $a$ is an element of $A$ such $a\in \varpi\h_{\mf{X},\zeta}$ for all $\zeta$ in $I(\mf{X})$ then $a\in \varpi A$. Consider $\varpi^{-1}a$ as an element of $A_K$. Note then that by assumption we have that $\varpi^{-1}a\in \h_{\mf{X},\zeta}\subseteq k(x)^+$ for all $x$ in $D(\mf{X})$. Thus, we see from \eqref{eq:spectral-seminorm-formula} and \cite[Proposition 6.2.3/1]{BGR} that $\varpi^{-1}a$ is in $A_K^\circ =A$ as desired. 

To see that $(R'_1)$ holds, passing to an affine open neighborhood of a given $\zeta\in I(\mf{X})$ we may assume that $\mf{X}=\Spf(A)$ and $|\mf{X}|$ is irreducible. Moreover, since the equality $A_K^\circ=A$ shows that $\Spec(\widetilde{A}_K)=\Spec(A_k)_\mathrm{red}$ (where $\widetilde{A}=A_K^\circ/A_K^\cc$ as in  \cite[\S1.2.5]{BGR}) we deduce that $\widetilde{A}$ is a domain, and since $A$ is moreover reduced, the supremum semi-norm $\|\cdot\|_\mathrm{sup}$ is a multiplicative norm by \cite[Proposition~6.2.3/5]{BGR}. This implies that $A$ is a domain; by the same argument, for every $f\in A$ with $f(\zeta)\neq 0$, the complete localization $A\langle f^{-1}\rangle$ is a domain, and passing to the limit also $\cO_{\mf{X}, \zeta}$ is a domain. To check that $\cO_{\mf{X}, \zeta}$ is a valuation ring, take a nonzero fraction $f/g \in \op{Frac}(\cO_{\mf{X}, \zeta})$; shrinking $\mf{X}$ further, we may assume $f,g\in A$ and $f^{-1}, g^{-1}\in A_K$. For the latter claim, since $A=A_K^\circ$, we can by \cite[Proposition 6.2.1/4 (ii)]{BGR} write $f^n = \alpha u$, $g^m=\beta v$ where $\alpha,\beta\in \cO_K$ and $u, v$ have supremum norm one and hence their images in $\widetilde{A}$ are nonzero; passing to a smaller affine neighborhood of $\zeta$, we may assume that $u$ and $v$ are invertible in $\widetilde{A}$ and hence in $A$ (because $A$ and $\widetilde{A}$ have the same maximal ideals), so that $f$ and $g$ become invertible in $A_K$. Then if say $\|f\|_{\rm sup}\leqslant \|g\|_{\rm sup}$, then $f/g\in A_K$ and $\|f/g\|_{\rm sup}\leqslant 1$, so $f/g \in A_K^\circ = A$ and so $f/g\in\h_{\mf{X},\zeta}$. Otherwise, $g/f\in \h_{\mf{X},\zeta}$ by the same argument.
    
Conversely, suppose that the conditions in (b) hold.  On the one hand, again by \eqref{eq:spectral-seminorm-formula} we have
\begin{equation*} 
    A_K^\circ=\left\{a\in A_K \,:\, a\in k(x)^+\text{ for each }x\in D(\mf{X})\right\}
\end{equation*}
where we are implicitly identifying $a\in A_K$ with its image in $k(x)$. On the other hand, we claim that 
\begin{equation*} 
    A=\left\{a\in A_K \,:\, a \in \h_{\mf{X},\zeta}\text{ for all }\zeta\in I(\mf{X})\right\}
\end{equation*}
where we identify $a\in A_K$ with its image in $\h_{\mf{X},\zeta}[\frac{1}{\varpi}]$. This would imply the required assertion by Lemma \ref{lem:r1'-equivalence} and the $(R'_1)$ property. 

To show the claim, we note that evidently $A$ maps into $\h_{\mf{X},\zeta}$. Conversely, let $a\in A_K$ be an element of the right hand side, and write $a=\varpi^{-n}a'$ with $a'\in A$ and $n\geqslant 0$. By the $(S'_2)$ property of $\mf{X}$, or equivalently the $(S'_1)$ property of $\Spec(A/\varpi^n)$, our assumption that $a'$ lies in $\varpi^n\cO_{\mf{X},\zeta}$ for all $\zeta$ implies that $a'\in \varpi^n A$, and hence $a\in A$ as desired. 
\end{proof}

Using the criterion of Proposition~\ref{prop:serre-crit}, Lemma~\ref{prop:etale-s2-prime} and Lemma~ \ref{lem:etale-r1-prime} imply the following result.

\begin{cor}\label{prop:etale-eta-normal}
    Let $\mf{Y}\to \mf{X}$ be an \'etale morphism of admissible formal schemes over $\cO_K$. If $\mf{X}$ is $\eta$-normal, then so is $\mf{Y}$.
\end{cor}

\bibliographystyle{amsalpha}
\renewcommand{\MR}[1]{MR \href{http://www.ams.org/mathscinet-getitem?mr=#1}{#1}}

\providecommand{\bysame}{\leavevmode\hbox to3em{\hrulefill}\thinspace}
\providecommand{\MR}{\relax\ifhmode\unskip\space\fi MR }
\providecommand{\MRhref}[2]{%
  \href{http://www.ams.org/mathscinet-getitem?mr=#1}{#2}
}
\providecommand{\href}[2]{#2}

\end{document}